\documentclass[reqno]{amsart}
\usepackage{amssymb,amscd,verbatim, amsthm,graphicx, color}
\usepackage{fancybox}
\usepackage{longtable}

\newcommand \fk[1]{{{\mathfrak #1}}}
\newcommand \C[1]{{\mathcal #1}}

\newcommand \wti[1]{{\widetilde {#1}}}

\newcommand \bC{{\mathbb C}}
\newcommand \bF{{\mathbb F}}

\newcommand \bR{{\mathbb R}}
\newcommand \bZ{{\mathbb Z}}

\newcommand \bQ{{\mathbb Q}}

\newcommand\one{1\!\!1}

\newcommand\CH{{\C H}}

\newcommand\CO{{\C O}}

\newcommand\CR{{\C R}}

\newcommand\ep{{\epsilon}}

\newcommand\al{{\alpha}}

\newtheorem{theorem}{Theorem}[section]
\newtheorem{conjecture}[theorem]{Conjecture}
\newtheorem{corollary}[theorem]{Corollary}
\newtheorem{lemma}[theorem]{Lemma}
\newtheorem{proposition}[theorem]{Proposition}

\theoremstyle{definition}
\newtheorem{definition}[theorem]{Definition}
\newtheorem{remark}[theorem]{Remark}
\newtheorem{example}[theorem]{Example}

\newcommand\Hom{\operatorname{Hom}}
\newcommand\End{\operatorname{End}}
\newcommand\Ext{\operatorname{Ext}}
\newcommand\Ind{\operatorname{Ind}}
\newcommand\ind{\operatorname{ind}}

\newcommand\tr{\operatorname{tr}}
\newcommand\pr{\operatorname{pr}}

\newcommand\triv{\mathsf{triv}}
\newcommand\sgn{\mathsf{sgn}}
\newcommand\refl{\mathsf{refl}}

\newcommand\gen{\mathsf{gen}}
\newcommand\ds{\mathsf{DS}}

\newcommand\el{\mathsf{ell}}
\newcommand\EP{\mathsf{EP}}
\newcommand\even{\mathsf{even}}
\newcommand\odd{\mathsf{odd}}
\newcommand\Irr{\mathsf{Irr}}

\newcommand\class{\mathsf{class}}

\newcommand\Res{\mathsf{Res}}

\newcommand\cyc{\mathsf{cyc}}

\newcommand\uni{\mathsf{un}}

\newcommand\Pin{\mathsf{Pin}}

\newcommand\iso{\mathsf{iso}}

\newcommand\EF{\mathsf{EF}}
\newcommand\Ell{\mathsf{Ell}}

\numberwithin{equation}{subsection}

\begin{document}

\title[Formal degrees and the exotic Fourier transform]{Formal degrees of unipotent discrete series representations
  and the exotic Fourier transform}

\author{Dan Ciubotaru}
        \address[D. Ciubotaru]{Mathematical Institute\\University of Oxford\\Andrew Wiles Building\\
Radcliffe Observatory Quarter\\Woodstock Road \\Oxford\\OX2 6GG\\UK}
       \email{dan.ciubotaru@maths.ox.ac.uk}

\author{Eric M.~Opdam}
\address[E.M. Opdam]{Korteweg-de Vries Institute for Mathematics\\Universiteit van Amsterdam\\Science Park 904\\ 1098 XH Amsterdam, The Netherlands}
\email{e.m.opdam@uva.nl} 


\thanks{This research was supported in part by NSF-DMS 1302122, NSA-AMS
  111016, and ERC-advanced grant no. 268105. We thank Roman Bezrukavnikov, George Lusztig,
  Eric Sommers, and
  David Vogan for helpful discussions. We also thank the referee for the careful and detailed review of the paper.}

\begin{abstract}
We introduce a notion of elliptic fake degrees for
unipotent elliptic representations of a semisimple $p$-adic group. We
conjecture, and verify in some cases, that the relation between the formal degrees of unipotent
discrete series representations of a semisimple $p$-adic group and
the elliptic fake degrees is given by the exotic Fourier
transform matrix introduced by Lusztig in the study of
representations of finite groups of Lie type.
\end{abstract}

\maketitle

\setcounter{tocdepth}{1}

\begin{small}
\tableofcontents
\end{small}

\section{Introduction}\label{s:1}

\subsection{}In this paper, we introduce a notion of elliptic fake degrees for
unipotent elliptic representations of a semisimple $p$-adic group. We
conjecture (Conjecture \ref{conj-main} and (\ref{conj-equiv})), and verify in some cases, that the relation between formal degrees of unipotent
discrete series representations of a semisimple $p$-adic group and
the elliptic fake degrees is given by the exotic Fourier
transform matrix introduced by Lusztig \cite{L1} in the study of
representations of finite groups of Lie type. In other words, we
expect that the picture is analogous to that for finite Lie groups,
where the degrees of the unipotent representations are related via the
exotic Fourier transform to their fake degrees.

The formal degrees of discrete series representations of semisimple $p$-adic groups
admit a conjectural description in terms of adjoint gamma factors
\cite{HII,GR}. For unipotent discrete series representations of a
semisimple $p$-adic group of adjoint type, this
conjecture has been verified in \cite{O3,O4}. (For exceptional split
groups, the formula for formal degrees had been known already by \cite{Re3}.) We hope that the present paper
will offer a new perspective on formal degrees from the ``arithmetic side''.

As we explain in section \ref{sec:ell}, see in particular Corollary
\ref{fake-ell-spin}, the elliptic fake degrees admit an interpretation in terms of
characters for the pin double cover of the finite Weyl group, in the
approach of \cite{C,CH,CT}. In this interpretation, the elliptic fake degrees are related to
the spin fake degree polynomials of \cite{BW}. They are also related
to certain specializations of the invariants  considered  in
\cite{GNS} and \cite{Som}.

It would be interesting to investigate how the elliptic fake degrees
fit into the new theory of unipotent almost characters of semisimple $p$-adic
groups initiated in \cite{KmL,L4,L5} and \cite{BKO}.

\subsection{} We give a brief outline of the paper. In section \ref{sec:ell}, we
define the elliptic fake degrees for elliptic representations of the
finite and (extended) affine Weyl group, and we compute them
explicitly for every irreducible Weyl group. The results for
exceptional groups are tabulated in Appendix \ref{sec:exc}. In section \ref{s:2},
we recall basic constructions and definitions for unipotent
representations and Deligne-Lusztig characters of finite groups of Lie type, in particular, the
definition (Definition \ref{d:exotic-Fourier}) and properties of the exotic Fourier transform. In section
\ref{sec:unip-padic}, we present our main conjecture \ref{conj-main},
as well as certain implications of it, in particular formulas
(\ref{conj-equiv}) and 
(\ref{e:conj-second}). We also give there examples in support of the
conjecture. Finally, in section \ref{sec:Iwahori}, we rephrase the
conjecture (Conjecture \ref{conj-ell-a}) in the case of representations with Iwahori fixed vectors via
the homological algebra interpretation of formal degrees from \cite{OS1}.

\section{Elliptic fake degrees}\label{sec:ell}

\subsection{} Let $E$ be the $l$-dimensional (real) reflection representation of the
finite Weyl group $W$. Let $\langle~,~\rangle_W$ denote the character
pairing on the Grothendieck group $R(W)$. Denote 
$${\wedge}_{s} E=\sum_{i\ge 0} s^i {\wedge}^i E,$$
an $s$-graded $W$-representation whose $W$-character is
$w\mapsto{\det}_E(1+sw)$, and
$$S_tE=\sum_{i\ge 0} t^i S^iE,$$
a $t$-graded $W$-representation whose $W$-character is $w\mapsto \frac 1{{\det}_E(1-tw)}.$

The elliptic pairing of $W$ is $\langle~,~\rangle^\el_W: R(W)\times R(W)\to \bZ$:
\begin{equation}
\begin{aligned}
\langle\sigma,\sigma'\rangle^\el_W=\langle \sigma,\sigma'\otimes
\wedge_{-1} E\rangle_W=\frac 1{|W|}\sum_{w\in
  W}\sigma(w)\sigma'(w){\det}(1-w).
\end{aligned}
\end{equation}

The radical of the form $\langle~,~\rangle^\el_W$ equals the lattice
$R_{\ind}(W)$ 
of parabolically induced $W$-characters, \cite[section 2.1]{R}. Let $\overline R(W)=R(W)/\text{rad}\langle~,~\rangle^\el_W$ be the
space of virtual elliptic representations of $W$. 

An element $w\in W$ is called elliptic if $\det(1-w)\neq 0.$ It is clear that the subset of elliptic elements of $W$ is closed under conjugation. A conjugacy class in $W$ is called elliptic if it consists of elliptic elements.

\begin{proposition}\label{p:ind-ell}
Suppose that $W$ is irreducible of classical type. The set of rational functions
$\left\{\frac 1{\det(1-qw)}\right\}\subset \bQ(q)$, where $w$ varies over a set
of representatives of elliptic conjugacy classes in $W$, is
$\bZ$-linear independent.
\end{proposition}

\begin{proof}
We verify the claim case by case as part of subsections \ref{sec:A-ell}-\ref{sec:ell-D}.
\end{proof}

\begin{remark}
The claim in Proposition \ref{p:ind-ell} is not necessarily true for
exceptional Weyl groups. In $F_4$,
there are $9$ elliptic conjugacy classes, see \cite[Table 8]{Ca2}. The
two classes labelled $D_4$ and $C_3\times A_1$ have the same
characteristic polynomial $\det(1-qw)=(q^3+1)(q+1)$.
\end{remark}

\begin{corollary}\label{c:ind-ell}
When $W$ is classical, the map $\langle~,\frac 1{\det(1-q\cdot~)}\rangle_W^\el: \overline
R(W)\to \bQ(q)$ is injective.
\end{corollary}

\begin{proof}
This is immediate from Proposition \ref{p:ind-ell}, since the form
$\langle~,~\rangle_W^\el$ is nondegenerate on $\overline R(W).$
\end{proof}

\subsection{}We define elliptic fake degrees for representations of
the finite Weyl group.
\begin{definition}\label{d:ell-fake-finite}
For every class $[\pi]\in \overline R(W)$, we define the elliptic fake
degree of $[\pi]$ to be
\begin{equation}
F_{[\pi]}=(q-1)^l\langle\pi, S_q
E\rangle^\el_W=\frac{(q-1)^l}{|W|}\sum_{w\in W} \pi(w)
\frac{\det(1-w)}{\det(1-qw)}\in \bQ(q).
\end{equation}
When we wish to emphasize the group $W$, we write $F_{[\pi]}^W.$
\end{definition}

\begin{remark}
Since $\wedge_{-1}E\otimes\sgn=(-1)^l\wedge_{-1}E,$ we have $\langle\sigma\otimes\sgn,\sigma'\rangle_W^\el=(-1)^l\langle\sigma,\sigma'\rangle_W^\el.$
\end{remark}

\begin{remark}
The elliptic fake degree of $[\sgn]$ is given by the formula
\begin{equation}
F_{[\sgn]}=(1-q)^l\prod_{i=1}^l\frac{1-q^{m_i}}{1-q^{m_i+1}},
\end{equation}
where $m_i$ are the exponents of $W$, see 
\cite[Chapter 5.5, Ex. 3]{Bou}.
\end{remark}

\subsection{} Consider the Clifford algebra $C(E)$ generated by $E$
with respect to the $W$-invariant product on $E.$ The defining
relation for $C(E)$ is 
\begin{equation}
\xi\cdot\xi'+\xi'\cdot\xi=-2(\xi,\xi'),\ \xi,\xi'\in E.
\end{equation}
Let $C(E)=C(E)_\even+C(E)_\odd$ be the natural $\bZ/2\bZ$-grading of
$C(E)$. 
Let $^t$ be the transpose anti-automorphism of $C(E)$ defined by
$\xi^t=-\xi,$ $\xi\in E$, and $\ep:C(E)\to C(E)$ the automorphism
which is $+1$ on $C(E)_\even$ and $-1$ on $C(E)_\odd.$ The pin group is
\begin{equation}
\Pin(E)=\{a\in C(E)^\times: \ep(a)\cdot E\cdot a^{-1}\subset E,\ a^t=a^{-1}\}.
\end{equation}
Then $p:\Pin(E)\to O(E)$, $p(a)(\xi)=\ep(a)\cdot\xi\cdot a^{-1}$ is a
two-fold cover of $O(E).$ Define the pin cover of $W$:
\begin{equation}
\wti W=p^{-1}(W)\subset \Pin(E).
\end{equation}
Set
\begin{equation}
\wti W'=\begin{cases}\wti W, &\dim E \text{ odd}\\\wti W\cap
  C(E)_\even, &\dim E\text{ even.}\end{cases}
\end{equation}
Let $S^\pm$ denote the two $\wti W$-modules obtained from the
restriction of the basic simple spin modules of $C(E).$ Let $R(\wti
W')_\gen$ denote the $\bZ$-span of irreducible genuine $\wti W'$-representations.

\begin{proposition}[{\cite[Theorem 4.2]{COT}}] The linear map $\iota: \overline R(W)\to R(\wti W')_\gen$,
  $\iota(\sigma)=\sigma\otimes (S^+-S^-)$ is an injective linear map
  such that
\begin{equation}
2\langle\sigma,\sigma'\rangle^\el_W=\langle\iota(\sigma),\iota(\sigma')\rangle_{\wti
W'}.
\end{equation}
For every $[\pi]\in \overline R(W)$, there exist unique associate
orthogonal $\wti W'$-representations $\wti\sigma^\pm_\pi$ such that
$$\iota([\pi])=\wti\sigma^+_\pi-\wti\sigma^-_\pi.$$ Moreover, if $[\pi]$ has
norm $1$ in $\overline R(W),$ then $\wti\sigma^\pm_\pi$ are irreducible. 
\end{proposition}

This result allows us to rewrite the elliptic fake degree in the
following form.
\begin{corollary}\label{fake-ell-spin} The elliptic fake degree equals
\begin{equation}\label{e:efd-spin}
F_{[\pi]}=\frac 12 (q-1)^l\langle
\wti\sigma_\pi^+-\wti\sigma_\pi^-,S_qE\otimes (S^+-S^-)\rangle_{\wti W'}.
\end{equation}
\end{corollary}

\subsection{} The irreducible
$S_n$-characters are parameterized by partitions $\lambda$ of $n$. We
write $\sigma_\lambda$ for the irreducible $S_n$-representation. In
this notation, $\sigma_{(n)}=\triv$ and $\sigma_{(1,1,\dots,1)}=\sgn.$
When $\lambda=(\lambda_1,\lambda_2,\dots,\lambda_k)$ is viewed as a
left justified decreasing Young diagram, and $(i,j)$ a box in
$\lambda$, denote
$$|\lambda|=k,\ n(\lambda)=\sum_{i\ge 1} (i-1)\lambda_i,\ c(i,j)=j-i,$$
and $h(i,j)$ the hook length of $(i,j).$ 

The conjugacy classes of $S_n$ are parameterized via the cycle decomposition by partitions $\al$
of $n$. Denote by $C^{S_n}_\al$ the conjugacy class.

We will need the following
combinatorial formula.

\begin{proposition}[{\cite[\S I.2, Ex. 5, \S I.3, Ex. 3]{Mac},
    cf. \cite[Proposition 3.1]{GNS}}] Let $E_n=\bC^n$ be the permutation
  representation of $S_n$.
The multiplicity of $\sigma_\lambda$ in $S_tE_n\otimes\wedge_s E_n$ equals
\begin{equation}\label{KP-comb}
g_\lambda(t,s)=\langle \sigma_\lambda,S_tE_n\otimes\wedge_s E_n\rangle_{S_n}=\frac{t^{n(\lambda)} \prod_{(i,j)\in\lambda} (1+s t^{c(i,j)})}{\prod_{(i,j)\in\lambda}(1-t^{h(i,j)})}.
\end{equation}
\end{proposition}

\subsection{}\label{sec:A-ell} Suppose $W$ is of type $A_{n-1}$.
The quotient $\overline R(S_n)$
has rank $1$ and it is spanned by $[\sgn].$ The only elliptic
conjugacy class in $S_n$ is the $n$-cycles, which has size
$(n-1)!$. In particular, Proposition \ref{p:ind-ell} is obvious in
this case. We have:
\begin{equation}
\begin{aligned}
F_{[\sgn]}&=\frac{(q-1)^{n-1}}{n!} \sgn((1,2,\dots,n))
\frac{\det(1-(1,2,\dots,n))}{\det(1-q(1,2,\dots,n))} (n-1)!
=\frac{(1-q)^n}{1-q^n},
\end{aligned}
\end{equation}
where we used that $\det(1-q(1,2,\dots,n))=1+q+q^2+\dots+q^{n-1}.$

\subsection{} Suppose $W$ is of type $B_n$. The elliptic conjugacy
classes $C_\al$ are indexed by partitions $\alpha=(\al_1,\dots,\al_k)$
of $n$. Regard $W(B_{\al_1})\times\dots\times W(B_{\al_k})$ naturally
as a subgroup of $W(B_n).$ Then a representative for the conjugacy
class $C^{B_n}_\al$ is $w_\al=\prod w_{\al_i}$, where $w_{\al_i}$ is a
Coxeter element for $W(B_{\al_i}).$  

The irreducible $W(B_n)$-representations are parameterized by pairs of
partitions $(\lambda,\gamma)$ of total size $n$. Let
$\lambda\times\gamma$ denote the irreducible representation. Then
$n\times\emptyset$ is the trivial representation and
$(\lambda\times\gamma)\otimes\sgn=\gamma^t\times\lambda^t.$

\begin{lemma}\label{l:Bn-on}
An orthonormal basis for $\overline R(W(B_n))$ is given by
$\{[\lambda\times\emptyset]:\lambda\vdash n\}$. More precisely, the
map $\sigma_\lambda\mapsto [\lambda\times\emptyset]$ induces an
isomorphism between the lattices $(R(S_n),\langle~,~\rangle_{S_n})$
and $(\overline R(W(B_n)),\langle~,~\rangle_{W(B_n)}^\el).$
\end{lemma}

\begin{proof}
We compute:
\begin{equation}\label{e:Bn-ell}
\begin{aligned}
\langle\lambda\otimes\emptyset,\lambda'\otimes\emptyset\rangle^\el_W&=\frac
1{|W|}\sum_{\al\vdash n} (\lambda\times\emptyset)(C^{B_n}_\al) \det(1-C^{B_n}_\al)
|C^{B_n}_\al|\\
&=\frac 1{n! 2^n}\sum_{\al\vdash n}\sigma_\lambda(\al) 2^{|\al|}
2^{n-|\al|} |C^{S_n}_\al|\\
&=\langle\sigma_\lambda,\sigma_{\lambda'}\rangle_{S_n}.
\end{aligned}
\end{equation}
This proves the claim.
\end{proof}

\begin{proposition}[{cf. \cite[Proposition 3.3]{GNS}}]\label{p:Bn-fake}
The elliptic fake degree of $[\lambda\times \emptyset]\in \overline
R(W(B_n))$ is
\begin{equation}
F_{[\lambda\times\emptyset]}^{B_n}=\displaystyle{(q-1)^n
  q^{2n(\lambda)}\prod_{(i,j)\in\lambda}
  \frac{1-q^{2c(i,j)+1}}{1-q^{2h(i,j)}}.}
\end{equation}
\end{proposition}

\begin{proof} Calculating as in (\ref{e:Bn-ell}), we find
\begin{equation}
F_{[\lambda\times\emptyset]}^{B_n}=(q-1)^n\langle\sigma_\lambda,c^{B_n}_q\rangle_{S_n},\text{
  where } c^{B_n}_q(C^{S_n}_\al)=\frac 1{\det_{B_n}(1-q C^{B_n}_\al)}=\frac
1{\prod_{i=1}^{|\al|} (1+q^{\al_i})}.
\end{equation}
Notice that $c^{B_n}_q(C^{S_n}_\al)=\frac {\prod_i
  (q^{\al_i}-1)}{\prod_i (q^{2\al_i}-1)}.$ Since $\det_{E_n}(1-q
C^{S_n}_\al)=\prod_i(q^{\al_i}-1)$, it follows that
$c^{B_n}_q$ is just the $S_n$-character of $S_{q^2}E_n\otimes \wedge_{-q}E_n.$
Thus the formula follows by applying (\ref{KP-comb}) with $s=-q$ and $t=q^2.$
\end{proof}

\begin{remark} Let $u\in G$ be a unipotent element and denote by $Z_G(u)$ the centralizer of $u$ in $G$ with identity component $Z_G(u)^0.$ The A-group is the finite group $A(u)=Z_G(u)/Z_G(u)^0Z(G)$ of components of $Z_G(u)$ modulo the center $Z(G)$ of $G$. Let $\C B_u$ denote the variety of Borel subgroups containing $u$ and denote by $d_u$  its complex dimension. As shown by Springer \cite{Sp}, the cohomology groups (with complex coefficients) $H^\bullet(\C B_u)$ admit an action of $A(u)\times W$. Let $\widehat{A(u)}$ denote the set of (isomorphism classes of) irreducible representations of $A(u)$ and set 
\begin{equation}\label{e:Springer-type}
\widehat{A(u)}_0=\{\phi\in\widehat{A(u)}:~H^{2d_u}(\C B_u)^\phi\neq 0\}.
\end{equation}

In the formulation of Conjecture \ref{conj-main}, an essential role is
played by the elliptic fake degrees of the $W$-modules given by the
Springer representations on $H^\bullet(\C B_u)^\phi$, see
(\ref{nonzero-fake}). Here $u$ is a quasi-distinguished unipotent
element (in the sense of \cite[(3.2.2)]{R}) in the complex group $G^\vee=Sp(2n,\bC)$ or $Spin(2n+1,\bC)$,
and $\phi\in \widehat{A(u)}_0.$ By \cite[Theorem 1.3 and Proposition A.6]{CH}, there exists a
unique 
partition $\lambda$ of $n$ such that
\begin{equation}
(\sum_{i=0}^{d_u} (-1)^{d_u-i} H^{2i}(\C B_u)^\phi)\otimes
(S^++S^-)=(\lambda\times\emptyset)\otimes (S^++S^-).
\end{equation}
For the explicit combinatorial procedure (based on algorithms of
Slooten \cite{Sl}) for attaching $\lambda$ to
$(e,\phi)$, see \cite[section 3.7]{C}. Applying the identity to $ww_0$, where $w$ is
an elliptic element of $W$ and $w_0$ is the long Weyl group element
and using \cite[Theorem 1.2]{CH}, we find that
\begin{equation}
H^\bullet(\C B_u)^\phi(w)=\ep(u,\phi)\cdot (\lambda\times\emptyset)(w),
\end{equation}
where $\ep(u,\phi)\in\{\pm 1\}$ is the sign of the scalar by which
$w_0$ acts on the irreducible Springer representation $H^{2d_u}(\C
B_u)^\phi$. In general, for affine Hecke algebras with arbitrary real parameters, the role of such signs in the relation between elliptic
theory and Dirac induction is investigated as part of \cite{CO}.
\end{remark}

We now verify Proposition \ref{p:ind-ell}. We need to show that the
functions
\begin{equation}
f_\al(q)=\frac 1{\prod_{i=1}^{|\al|} (1+q^{\al_i})},
\quad\al=(\al_1,\al_2,\dots)\text{ (decreasing) partition of }n,
\end{equation}
are $\bZ$-linear independent. Suppose by induction that the claim
holds for all $m<n$. For $n$, divide the set of functions into subsets
indexed by $k=1,\dots,n$,
\begin{equation}
\C F_k=\{f_\al(q):\ \al\vdash n,\ \al_1=k\}.
\end{equation}
By induction, each subset $\C F_k$ is $\bZ$-linear independent (using
the induction hypothesis for $m=n-k$). Moreover, it is easily seen
that, for $a^k_\al\in\bZ$,
\begin{equation}
\sum_{k=1}^n \sum_{f_\al(q)\C F_k} a^k_\al f_\al(q)=0\text{ implies }
\sum_{f_\al(q)\C F_k} a^k_\al f_\al(q)=0,\text{ for every }k.
\end{equation}
For example, one first multiplies by $1+q^n$ and specializes $q$ to a
primitive $2n$-th root of $1$, then multiply by $1+q^{n-1}$, etc.

\subsection{}\label{sec:ell-D} Suppose $W$ is of type $D_n$. The group $W(D_n)$ is a
natural subgroup of $W(B_n).$ The elliptic element $w_\al$ in $W(B_n)$
lives in $W(D_n)$ if and only if $\al$ is a partition with an even
number of parts. The set $\{w_\al:\al\vdash n, |\al| \text{ is
  even}\}$ is a complete set of representatives for the elliptic
conjugacy classes in $W(D_n)$.

The irreducible $W(D_n)$ representations are obtained by restriction
from $W(B_n)$ as follows. If $\al\times\beta$ is an irreducible
$W(B_n)$-representation with $\al\neq \beta$, then the restrictions
$(\al\times\beta)|_{D_n}\cong (\beta\times\al)|_{D_n}$ are irreducible
$W(D_n)$-representations. If $\al=\beta,$ then $\al\times\al$
restricted to $W(D_n)$ decomposes into two inequivalent
equidimensional representations, $(\al\times\al)^I$ and $(\al\times\al)^{II}.$

Notice that the defining representation $E$ for $D_n$ is the same as
the one for $B_n$, therefore we may use Frobenius reciprocity to see
that 
\begin{equation}\label{e:Frob}
\langle \sigma,\Res^{B_n}_{D_n}\sigma'\rangle^\el_{W(D_n)}=\langle\Ind_{D_n}^{B_n}\sigma,\sigma'\rangle^\el_{W(B_n)}.
\end{equation}

\begin{lemma}\label{l:Dn-on} Let $\overline P(n)$ denote the set classes of
  partitions of $n$ under the relation $\lambda\cong\lambda^t.$
\begin{enumerate}
\item Suppose $n$ is odd. The set $\{[\lambda\times\emptyset]:\lambda\in\overline P(n),
  \lambda\neq \lambda^t\}$ is an orthonormal basis for $\overline
  R(W(D_n)).$
\item Suppose $n$ is even. The set
  $\{[\lambda\times\emptyset]:\lambda\in\overline P(n)\}$ is an
  orthogonal basis for $\overline R(W(D_n))$. Moreover,
$$\langle \lambda\times\emptyset,\lambda\times\emptyset\rangle^\el_{W(D_n)}=\begin{cases}1,&\lambda\neq\lambda^t,\\2,&\lambda=\lambda^t.\end{cases}$$
\end{enumerate}
\end{lemma}

\begin{proof} Suppose $\lambda,\lambda'\in \overline P(n).$ 
By (\ref{e:Frob}), we have $\langle
\lambda\times\emptyset,\lambda'\times\emptyset\rangle^\el_{W(D_n)}=\langle
\lambda\times\emptyset+\emptyset\times\lambda,\lambda'\times\emptyset\rangle^\el_{W(B_n)}=\langle
\lambda\times\emptyset,\lambda'\times\emptyset\rangle^\el_{W(B_n)}+(-1)^n\langle
\lambda^t\times\emptyset,\lambda'\times\emptyset\rangle^\el_{W(B_n)}.$
The claims now follow easily from Lemma \ref{l:Bn-on}.
\end{proof}

\begin{proposition}\label{p:Dn-fake}
The elliptic fake degree of $[\lambda\times \emptyset]\in \overline
R(W(D_n))$ is
\begin{equation}
\begin{aligned}
F_{[\lambda\times\emptyset]}^{D_n}=\displaystyle{\frac{(q-1)^n}{\displaystyle{\prod_{(i,j)\in\lambda}}(1-q^{2h(i,j)})}\left(
  q^{2n(\lambda)}{\prod_{(i,j)\in\lambda}}
  (1-q^{2c(i,j)+1})+  (-1)^nq^{2n(\lambda^t)}\prod_{(i',j')\in\lambda^t}
  (1-q^{2c(i',j')+1})\right).}
\end{aligned}
\end{equation}
\end{proposition}

\begin{proof} As in the proof of Lemma \ref{l:Dn-on}, we see that 
$$F_{[\lambda\times\emptyset]}^{D_n}=F_{[\lambda\times\emptyset]}^{B_n}+(-1)^n
F_{[\lambda^t\times\emptyset]}^{B_n},$$
and the formula follows from Proposition \ref{p:Bn-fake}.

\end{proof}

Since the functions $1/\det(1-qw)$, $w$ elliptic, in type $D_n$ are a
subset of the ones for type $B_n$, they are also $\bZ$-linear independent.

\subsection{}\label{sec:affine-ell} Let $\C R=(X,R,X^\vee, R^\vee,F)$ be a based root datum. In
particular, $X,X^\vee$ are lattices in perfect duality
$\langle~,~\rangle:X\times X^\vee\to\bZ$, $R\subset X\setminus\{0\}$
and $R^\vee\subset X^\vee\setminus\{0\}$ are the (finite) sets of
roots and coroots respectively, and $F\subset R$ is a basis of
simple roots. Let $W$ be the finite Weyl group with set of generators
$S=\{s_\al:\al\in F\}.$ Set $W^e=W\ltimes
X$, the extended affine Weyl group, and $W^a=W\ltimes Q$, the affine
Weyl group, where $Q$ is the root lattice of $R$. Then $W^a$ is
normal in $W^e$ and $\Omega:=W^e/W^a\cong X/Q$ is an abelian group. We assume that $\C R$ is semisimple, i.e., $\Omega$ is a finite group.

The set $R^a=R^\vee\times \bZ\subset X^\vee\times\bZ$ is the set of
affine roots. A basis of simple affine roots is given by
$F^a=(F^\vee\times\{0\})\cup\{(\gamma^\vee,1): \gamma^\vee\in R^\vee
\text{ minimal}\}.$ For every affine root $\mathbf a=(\al^\vee,n)$, let
$s_{\mathbf a}:X\to X$ denote the reflection
$s_{\mathbf a}(x)=x-((x,\al^\vee)+n)\al.$ The affine Weyl group $W^a$
has a set of generators $S^a=\{s_{\mathbf a}: \mathbf a\in F^a\}$.
Let $l:W^e\to\bZ$ be the length function. 

Set $E=X\otimes_\bR \bR$, so the discussion regarding elliptic theory of $W$ and $E$ from the previous sections applies. We denote a typical element of $W^e$ by $w t_x$, where $w\in W$ and $x\in X.$ The extended affine Weyl group $W^e$ acts on $E$ via $(w t_x)\cdot v=w\cdot v+ x,$ $v\in E.$

An element $w t_x\in W^e$ is called elliptic if $w\in W$ is elliptic (with respect to the action on $E$), or equivalently, if $wt_x$ has an isolated fixed point in $E$. For basic facts about elliptic theory for $W^e$, see \cite[sections 3.1, 3.2]{OS1}. There are finitely many elliptic conjugacy classes in $W^e$ (and in $W^a$). 

Let $W^e$-mod be the category of finite dimensional $W^e$-modules. Define the Euler-Poincar\'e pairing on $W^e$-mod as follows:
\begin{equation}
\langle U,V\rangle^\EP_{W^e}=\sum_{i\ge 0}(-1)^i\dim \Ext^i_{W^e}(U,V),\quad U,V\in W^e\text{-mod}.
\end{equation}
Let $R(W^e)$ be  the Grothendieck group of $W^e$-mod, and set $$\overline R(W^e)=R(W^e)/\text{rad}\langle~,~\rangle^\EP_{W^e}.$$ 
By \cite[Theorem 3.3]{OS1}, the Euler-Poincar\'e pairing for $W^e$ can also be expressed as an elliptic integral. More precisely, define the conjugation-invariant elliptic measure $\mu_\el$ on $W^e$ by setting $\mu_\el=0$ on nonelliptic conjugacy classes, and for an elliptic conjugacy class $C$ such that $v\in E$ is an isolated fixed point for some element of $C$, set
\begin{equation}
\mu_\el(C)=\frac{|Z_{W^e}(v)\cap C|}{|Z_{W^e}(v)|};
\end{equation}
here $Z_{W^e}(v)$ is the isotropy group of $v$ in $W^e.$ Then
\begin{equation}\label{affine-ell-pair}
\langle U, V\rangle^\EP_{W^e}=\langle \chi_U,\chi_V\rangle^\el_{W^e}:=\int_{W^e}\chi_U\chi_V~d\mu_\el,\ U,V\in W^e\text{-mod},
\end{equation}
where $\chi_U,\chi_V$ are the characters of $U$ and $V$.

Set $T^\vee=\Hom_\bZ(X,\bC^\times)$. (The superscript $\vee$ is so that the notation is consistent in a later section.) Then $W$ acts on $T^\vee$. For every $s\in T^\vee,$ set
\begin{equation}
W_s=\{w\in W: w\cdot s=s\},
\end{equation}
and one considers the elliptic theory of the finite group $W_s$ acting on the cotangent space of $T^\vee$ at $s$. By Clifford theory, consider the induction map
\begin{equation}
\Ind_s: W_s\text{-mod}\to W^e\text{-mod},\quad \Ind_s(U):=\Ind_{W_s\ltimes X}^{W^e}(U\otimes s),
\end{equation}
which maps irreducible modules to irreducible modules.

By \cite[Theorem 3.2]{OS1}, the map 
\begin{equation}\label{ind-iso}
\bigoplus_{s\in T^\vee/W}\Ind_s: \bigoplus_{s\in T^\vee/W}\overline R(W_s)_\bC\to \overline R(W^e)_\bC
\end{equation}
is an isomorphism of metric spaces, in particular,
\begin{equation}\label{ind-pair}
\langle \Ind_s U,\Ind_s V\rangle^\EP_{W^e}=\langle U,V\rangle^\el_{W_s},\ U,V\in W_s\text{-mod}.
\end{equation}
It is clear that the only nonzero contributions in the left hand side of (\ref{ind-iso}) comes from ``isolated'' elements of $T^\vee$, more precisely
\begin{equation}
T^\vee_\iso=\{s\in T^\vee: w\cdot s=s\text{ for some elliptic }w\in W\}.
\end{equation}
Conversely, suppose $\chi\in \overline R(W^e)_\bC$ is given. For every $s\in T^\vee_\iso/W$, one may project $\chi$ onto $\Ind_s\overline R(W_s)_\bC$. Call the projection $\pr_s\chi\in \Ind_s\overline R(W_s)_\bC.$ (Of course, this projection makes sense even for non-elliptic modules.) By the injectivity of (\ref{ind-iso}), there exists a unique element, which we denote $\overline\pr_s \chi\in \overline R(W_s)_\bC$ such that
\begin{equation}
\pr_s\chi=\Ind_s\overline\pr_s\chi.
\end{equation} 
With this notation, the inverse map is
\begin{equation}
\bigoplus_{s\in T^\vee_\iso/W}\overline \pr_s: \overline R(W^e)_\bC\to \bigoplus_{s\in T^\vee_\iso/W}\overline R(W_s)_\bC.
\end{equation}
This allows us to define elliptic fake degrees for elliptic $W^e$-modules.

\begin{definition}\label{d:ell-fake-affine}
For every class $[\pi]\in \overline R(W^e)$, define the elliptic fake degree of $[\pi]$ to be
$$F_{[\pi]}^{e}= F_{[\overline\pr_1\pi]},$$
where $F_{[\overline\pr_1\pi]}$ is the elliptic fake degree in the finite reflection group $W=W_s$, as in Definition \ref{d:ell-fake-finite}.

The discussion above also applies if we consider just $W^a$ instead of $W^e$, and in that case we denote the elliptic fake degree by $F_{[\pi]}^a.$ 
\end{definition}

\section{Unipotent representations of finite groups of Lie
  type}\label{s:2}

\subsection{} Let $p,\ell$ be primes, $\ell\neq p$, and $q$ be a power of 
$p$. Let $k$ be an algebraic closure of $\bF_p$. Let $G$ be a 
connected algebraic reductive $k$-group defined over $\bF_q$, and let $F:G\to G$
be the associated Frobenius homomorphism. 
Let $L:G\to G$,
$L(g)=g^{-1}F(g)$, be Lang's map. Let $\C C_\class(G^F)$ be the
$\overline\bQ_\ell$-vector space of class functions $G^F\to
\overline\bQ_\ell$ with the character pairing $\langle~,~\rangle.$ Let $\Irr G^F$ denote the set of irreducible
$\overline\bQ_\ell$-characters of $G^F.$


The Deligne-Lusztig generalized character $R_{T,\theta}\in \C
C_\class(G^F)$ is defined for an $F$-stable maximal torus $T$, and a
character $\theta: T^F\to \overline\bQ_\ell^*$ as follows. Let $B=TU$
be a Borel subgroup of $G$, and set $\wti X=L^{-1}(U).$ Then $\wti X$
carries a left $G^F$-action and a right $T^F$ action. Set
\begin{equation}
R_{T,\theta}(g)=\sum_{i\ge 0} (-1)^i \tr(g, H_c^i(\wti
X,\overline\bQ_\ell)_\theta),\quad g\in G^F;
\end{equation} 
here $H^i_c(\wti X,\overline\bQ_\ell)$ denotes the $i$-th $\ell$-adic
cohomology group with compact support, and the subscript $\theta$
indicates the $\theta$-isotypic component. By \cite[Theorem 4.2]{DL},
$R_{T,\theta}$ is independent of the choice of $B$.  

If $\rho\in \Irr
G^F$, there exists $(T,\theta)$ such that $\langle
\rho,R_{T,\theta}\rangle\neq 0.$ 

\begin{definition}
A character $\rho\in\Irr G^F$ is called unipotent
if $\langle\rho, R_{T,1}\rangle\neq 0$ for a maximal $F$-stable torus
$T$. One calls $\rho$ cuspidal if $\langle\rho,R_{T,\theta}\rangle=0$
for any $F$-stable maximal torus $T$ contained in some proper
$F$-stable parabolic subgroup of $G$ and any character $\theta$ of $T^F.$ Denote by $\Irr_{\uni} G^F$ the set of irreducible unipotent characters of $G^F.$
\end{definition}

\subsection{} Fix $B_0$ an $F$-stable Borel subgroup of $G$,
$B_0=T_0U_0$, and $T_0$ is a $F$-stable maximal torus. Let $W$ be the
Weyl group of $T_0$. For every $w\in W$, choose an $F$-stable
representative $\dot w$ in $G$ and $x\in G$ such that $x^{-1}F(x)=\dot
w.$ Denote $T_w=xT_0x^{-1}$, an $F$-stable maximal torus, and set 
\begin{equation}
R_w=R_{T_w,1},\ w\in W.
\end{equation}
The character $R_w$ admits the following alternative description:
\begin{equation}
R_w(g)=\sum_{i\ge 0} (-1)^i \tr(g, H_c^i(X_w,\overline\bQ_\ell)),\
g\in G^F,
\end{equation}
where $X_w=\{g B_0\in G/B_0: g^{-1} F(g)\in B_0\dot w B_0\}.$ In
particular, 
\begin{equation}\label{pair-Rw}
\langle R_w,R_{w'}\rangle=\#\{w_1\in W: w_1 w'=w F(w_1)\}.
\end{equation}
and 
\begin{equation}
R_w(1)=\sum_{i\ge 0} \tr (F\circ w, \overline S^i) q^i,
\end{equation}
where $S$ is algebra of $\bQ$-polynomial function of $\bQ\otimes
\Hom(k^*,T_0)$, and $\overline S$ is the graded algebra of
$W$-coinvariants.

\subsection{} From now on, we assume for simplicity that $G^F$ is split. For every $\C E\in \widehat W$, define the unipotent
almost-character
\begin{equation}
R_{\C E}=\frac 1{|W|}\sum_{w\in W}\tr (w,{\C E}) R_w.
\end{equation}
The general definition of the unipotent almost-character is in \cite[(3.7.1)]{L1}. In the split case,\begin{equation}
R_{\C E}(1)=\sum_{i\ge 0}\dim({\C E}\otimes \overline S^i)^W q^i,
\end{equation}
which means that $R_{\C E}(1)$ equals the fake degree
of ${\C E}$:
\begin{equation}
f_{\C E}(q)=\sum_{i\ge 0} \dim ({\C E}\otimes \overline S^i)^W q^i=(1-q)^r P(q)\frac
1{|W|}\sum_{w\in W} \frac{\tr(w, {\C E})}{\det_{\overline S^1}(1-q w)},
\end{equation}
where $r$ is the rank of $G^F$ and $P(q)$ is the Poincar\'e polynomial
of $W$.

The set $\{R_{\C E}: {\C E}\in \widehat W\}$ is orthonormal. This follows
from (\ref{pair-Rw}). Also from (\ref{pair-Rw}), it is easy to deduce
that if $g\in G^F$ is regular semisimple and $w_g\in W$ is the unique
Weyl group element such that $g\in T_{w_g}$, then
\begin{equation}
R_{\C E}(g)=\tr(w_g,{\C E}).
\end{equation}

\subsection{}For every ${\C E}\in \widehat W$, there exists a corresponding irreducible
unipotent $G^F$-representation $\rho_{\C E}$, as follows. The algebra $\End_{G^F}(\Ind_{B_0^F}^{G^F}(\triv))$ is
isomorphic to the finite Hecke algebra $\CH(W,q)$ of $W$. Recall that
$\CH(W,q)$ is the $\overline Q_\ell$-algebra (or $\bC$-algebra)
spanned by $\{t_w: w\in W\}$ subject by the relations:
\begin{equation}
\begin{aligned}
& t_s^2=(q-1) t_s+q t_1,\quad s \text{ simple reflection}.\\
&t_w\cdot t_{w'}=t_{ww'},\quad \text{ if }\ell(ww')=\ell(w)+\ell(w');\\
\end{aligned}
\end{equation}
here $\ell(w)$ denotes the length function of $W$. 
By Tits' deformation argument, $\CH(W,q)$ is isomorphic to the group
algebra of $W$, and to every ${\C E}$, there corresponds a simple
module of $\CH(W,q)$, and thus an irreducible unipotent
$G^F$-representation $\rho_{\C E}$ which occurs in
$\Ind_{B_0^F}^{G^F}(\triv).$

\subsection{}\label{sec:fourier} 
If $\Gamma$ is a finite group, consider the set of pairs
$\wti M(\Gamma)=\{(x,\sigma): x\in\Gamma,\ \sigma\in\Irr
C_\Gamma(x)\}.$ The group $\Gamma$ acts on $\wti M(\Gamma)$ by
$\gamma\cdot(x,\sigma)=(\gamma x \gamma^{-1},\sigma^\gamma),$ where
$\sigma^\gamma$ is the twist of $\sigma$ by $\gamma.$ Let $M(\Gamma)$
denote the set of orbits.
\begin{definition}[{\cite[(4.14.3)]{L1}}]\label{d:exotic-Fourier}
The exotic Fourier transform matrix associated to $\Gamma$ is the
square matrix of size $\# M(\Gamma)$ with entries
\begin{equation}
\{(x,\sigma),(y,\tau)\}=\frac 1{|C_\Gamma(x)| |C_\Gamma(y)|}\sum_{\substack{g\in
\Gamma\\x gyg^{-1}=gyg^{-1}x}}\sigma(gyg^{-1})\overline{\tau(g^{-1}x g)}.
\end{equation}
\end{definition}

 In \cite[chapter 4]{L1}, Lusztig partitioned the set of irreducible $W$-characters into families $\C F$, and attached to each family a finite group $\Gamma_{\C F}$ together with an injective map
 $\C F\hookrightarrow M(\Gamma_{\C F})$. Denote the image of this map by $M(\Gamma_{\C F})'$. Define the parameterizing set \cite[(4.21.1)]{L1}
 \begin{equation}
  X(W)=\displaystyle{\bigsqcup_{\C F\subset \Irr W} M(\Gamma_{\C F})},
 \end{equation}
and the pairing \cite[(4.21.2)]{L1}
\begin{equation}\label{pair-X}
\{~,~\}: X(W)\times X(W)\to \bQ,
\end{equation}
as follows: for $(x,\sigma)\in M(\Gamma_{\C F})$ and $(y,\tau)\in M(\Gamma_{\C F'})$,  $\{(x,\sigma),(y,\tau)\}$ is as in Definition \ref{d:exotic-Fourier} if $\C F=\C F'$, and otherwise it is zero.

Let 
\begin{equation}\label{Delta}
\Delta: M(\Gamma_{\C F})\to \{\pm 1\} 
\end{equation}
be the function defined in \cite[\S4.14 and \S4.21]{L1}. Recall that if $W$ is an irreducible Weyl group, then $\Delta$ is the constant function $1$, except for the families $\C F$ that contain the representation $512_a'$ in $E_7$, or $4096_z$ or $4096_x'$ in $E_8.$

\medskip

As in \cite[\S4.22]{L1}, we assume that $G$ has connected center. We recall next the main theorem of the classification of $\Irr_\uni G^F$ from \cite{L1} together with a character formula that we will need later in the paper.

\begin{theorem}[{\cite[Theorem 4.23, (4.26.1)]{L1}}] 
There exists a bijection $$X(W)\longleftrightarrow \Irr_\uni G^F,\quad (x,\sigma)\to \rho_{(x,\sigma)}$$ such that
for every $(x,\sigma)\in X(W)$ and ${\C E}\in
  \widehat W,$ we have
\begin{equation}
\langle\rho_{(x,\sigma)},R_{\C E}\rangle=\Delta(x,\sigma) \{(x,\rho),(y,\tau)\},
\end{equation}
where $(y,\tau)$ is such that $\rho_{(y,\tau)}=\rho_{\C E}.$
Moreover, if $g\in G^F$ is semisimple, then
\begin{equation}\label{char-semi}
\tr(g,\rho_{(x,\sigma)})=\sum_{(y,\tau)\in M(\Gamma)'} \Delta(x,\sigma)
\{(x,\sigma),(y,\tau)\} ~R_{(y,\tau)}(g);
\end{equation}
here we write $R_{(y,\tau)}$ in place of $R_{\C E}$ when $\rho_{(y,\tau)}=\rho_{\C E}.$
\end{theorem}

Suppose $(x,\sigma)\in M(\Gamma_{\C F})\subset X(W)$ is given. The two extreme cases of (\ref{char-semi}) are:
\begin{enumerate}
\item if $g\in G^F$ is regular semisimple, then
\begin{equation}
\tr(g,\rho_{(x,\sigma)})=\sum_{(y,\tau)\in M(\Gamma_{\C F})'} \Delta(x,\sigma)
\{(x,\sigma),(y,\tau)\} \tr(w_g,{\C E});
\end{equation}
\item if $g=1$, then the formal degree of an irreducible unipotent
  representation is
\begin{equation}\label{formal-finite}
\mu_{\rho_{(x,\sigma)}}:=\rho_{(x,\sigma)}(1)=\sum_{(y,\tau)\in M(\Gamma_{\C F})'} \Delta(x,\sigma)
\{(x,\sigma),(y,\tau)\} f_{\C E}(q);
\end{equation}
\end{enumerate}
here ${\C E}$ is such that $\rho_{\C E}=\rho_{(y,\tau)}.$

\section{Unipotent representations of reductive $p$-adic groups}\label{sec:unip-padic}

\subsection{} Let $K$ be a nonarchimedian local field, with ring of
integers $\C O$, prime ideal $\fk p$, and finite residue field
$\bF_q=\CO/\fk p.$  Let $G$ be a connected reductive algebraic group defined
over $K$, and $G(K)$ its $K$-points. If $P$ is a parahoric subgroup of
$G(K),$ let $U_P$ be its pro-unipotent radical, and $\overline
P=P/U_P$ be the reductive quotient, a reductive group over $\bF_q.$

\begin{definition}[\cite{L2}]
An irreducible smooth representation $(\pi,V)$ of $G(K)$ is called
unipotent if there exists a parahoric subgroup $P$ of $G(K)$ and a
cuspidal unipotent representation $\rho$ of $\overline P$ such that
$\Hom_{\overline P}[\rho,V^{U_P}]\neq 0.$
\end{definition}

Fix a Haar measure $\mu$ on $G(K).$ Suppose $(\rho,\C W)$ is a
cuspidal unipotent representation of $\overline P.$ Let $\wti\rho$ be
the pullback to $P$ of $\rho.$ Define $\C H(G(K),\rho)$ to be the
$\mu$-convolution algebra of compactly supported, smooth functions 
\begin{equation}
\C H(G(K),\rho)=\{f: G(K)\to \End(\C W): f(p_1xp_2)=\wti\rho(p_1)
f(x)\wti\rho(p_2),\ p_1,p_2\in P, x\in G(K)\},
\end{equation}
with unit $1_\rho=\frac 1{\mu(P)} \wti\rho \chi_P,$ where $\chi_P$ is
the characteristic function of $P$.

Denote by $\C R(G(K),\rho)$ the category of complex smooth
representations $V$ of $G(K)$ such that 
the $\wti\rho$-isotypical component $V^\rho\subset V|_P$ of the 
restriction of $V$  to $P$ generates $V$, i.e. $H(G(K))V^\rho=V$.
 The Hecke algebra $\C H(G(K),\rho)$ acts naturally on 
 $V_\rho:=\Hom_{P}[\wti\rho,V]=\Hom_{\overline{P}}[\rho,V^{U_P}]$, 
giving rise to a functor
\begin{equation}\label{functor-rho}
m_\rho: \C R(G(K),\rho)\to \C H(G(K),\rho)\text{-mod},\ V\mapsto V_\rho.
\end{equation}
In this case, it is known from \cite{Mo}  and \cite{MP} that $m_\rho$ is an equivalence of categories. 
This is a generalization of Borel's classical result \cite{B} in the case of Iwahori subgroups.

\subsection{}
The algebra $\C H(G(K),\rho)$ has a structure of a normalized
$\bC$-algebra with respect to the $*$-operation:
\begin{equation}
f^*(x)=f(x^{-1})^*,
\end{equation}
where the second $*$ means the conjugate transpose operation on $\C
W$, and inner product
\begin{equation}
[f_1,f_2]=\frac{\mu(P)}{\rho(1)} \tr ((f_1^*\star f_2)(1)).
\end{equation}
Since all simple $\C H(G(K),\rho)$-modules are finite dimensional, $\C
H(G(K),\rho)$ admits an abstract Plancherel formula. Let $C^*_r(\C H(G(K),\rho))$ denote the reduced $C^*$-algebra completion of $\C H(G(K),\rho)$, see for example \cite[\S3.1 and \S4.1]{BHK}.

\begin{theorem}[\cite{Dix}]
There exists a unique positive Borel measure $\hat\mu_\rho$ (depending
on $\mu$) such that
\begin{equation}
[f,1_\rho]=\int_{\widehat{C^*_r(\C H(G(K),\rho))}}
\tr\pi(f)~d\hat\mu_\rho(\pi),\quad f\in \CH(G(K),\rho).
\end{equation}
\end{theorem}

The transfer of Plancherel measures under the functor $m_\rho$ behaves
very well. More precisely, if we let $\widehat{C^*_r(G(K),\rho)}$
denote the support of the Plancherel measure of $G(K)$ in the
subcategory $\C R(G(K),\rho)$, then the following result holds.

\begin{theorem}[{\cite[Theorem B]{BHK}}] The functor $m_\rho$ induces a
  homeomorphism 
$$\hat m_\rho: \widehat{C^*_r(G(K),\rho)}\to \widehat{C^*_r(\C
  H(G(K),\rho))},$$
such that for every Borel set $S$ of $\widehat{C^*_r(G(K),\rho)}$, one
has
\begin{equation}
\hat\mu(S)=\frac{\rho(1)}{\mu(P)}~\hat\mu_\rho(\hat m_\rho(S)).
\end{equation}
\end{theorem}

\begin{example}\label{ex:Iwahori}
Suppose $P=I$ is an Iwahori subgroup, and $\rho=1_{\overline I}$ is the trivial
representation of $\overline I.$ Then $\C H(G(K),1_{\overline I})$ is
the Iwahori-Hecke algebra of $I$-biinvariant functions on $G(K).$
Normalize the Haar measure $\mu$ such that $\mu(I)=1.$ Then for every
irreducible discrete series representation $V$ of $G(K)$ such that
$V^I\neq 0,$ the formal degree is
\begin{equation}
\hat\mu(V)=\hat\mu_{1_{\overline I}}(V^I).
\end{equation}
\end{example}

In this way, the computation of formal degrees of discrete series
representations in $\C R(G(K),\rho)$ can be reduced to the similar
problem for $\C H(G(K),\rho)$-modules.

\subsection{} All the algebras $\C H(G(K),\rho)$ in the previous
subsection are specializations of affine Hecke algebras that we define next.

Let $\C R=(X,R,X^\vee, R^\vee,F)$ be a based root datum. We retain the notation from section \ref{sec:affine-ell}.

Let $\mathbf q=\{\mathbf q(s): s\in S^a\}$ be a set of invertible, commuting
indeterminates such that $\mathbf q(s)=\mathbf q(s')$ whenever $s,s'$ are
$W^a$-conjugate. Let $\Lambda=\bC[\mathbf q(s),\mathbf q(s)^{-1}: s\in S^a]$.

\begin{definition}[Generic affine Hecke algebra]\label{d:Hecke-generic} The generic affine
  Hecke algebra $\C H(\C R,\mathbf q)$ associated to the root datum
  $\C R$ and the indeterminates $\mathbf q$ is the unique associative,
  unital $\Lambda$-algebra with basis $\{N_w: w\in W^e\}$ and
  relations
\begin{enumerate}
\item[(i)] $N_w N_{w'}=N_{ww'},$ for all $w,w'\in W$ such that
  $l(ww')=l(w)+l(w')$;
\item[(ii)] $(N_s-\mathbf q(s))(N_s+\mathbf q(s)^{-1})=0$ for all $s\in S^a.$
\end{enumerate}
\end{definition}

Fix an indeterminate $\mathbf q.$ Given a $W^a$-invariant function $m:S^a\to \bR$, we may
define a homomorphism $\lambda_{m}:\Lambda\to \bC[\mathbf q]$, $\mathbf
q(s)=\mathbf q^{m(s)}$.  Consider the specialized affine Hecke algebra
\begin{equation}
\C H(\C R, m)=\C H(\C R,\mathbf q)\otimes_\Lambda \bC_{\lambda_{m}}.
\end{equation}

\subsection{}We return now to the setting of unipotent representations
of the $p$-adic group $G(K)$. Suppose that $G$ is simple of adjoint type and $G(K)$ is split. Let $\widehat{G(K)}^\uni$ denote the set of (equivalence classes of) irreducible unipotent representations of $G(K)$, and $\widehat{G(K)}^\uni_\ds$ the subset of irreducible unipotent discrete series representations. 

Let $G^\vee$ denote the complex dual group. 
An element $x\in G^\vee$ is called elliptic if the centralizer $Z_{G^\vee}(x)$ does not contain any nontrivial torus, or equivalently, the conjugacy class of $x$ does not meet any proper Levi subgroup of $G^\vee.$ For every $x\in G^\vee,$ define the A-group
\begin{equation}
A(x)=Z_{G^\vee}(x)/Z_{G^\vee}(x)^0 Z(G^\vee).
\end{equation}
Let $\widehat{A(x)}$ be the set of isomorphism classes of irreducible representations of $A(x)$.

The Deligne-Langlands-Lusztig classification for $\widehat{G(K)}^\uni$
takes the following form. In the case of representations with Iwahori
fixed vectors, it was proved in \cite[Theorem 7.12 and Theorem 8.3]{KL}.

\begin{theorem}[{\cite[Corollary 6.5(a)]{L2},\cite[\S10.9 and Theorem 10.11]{L3}}]\label{t:DLL} There exists a natural one-to-one correspondence
\begin{equation}
\widehat{G(K)}^\uni\longleftrightarrow G^\vee\backslash\{(x,\phi):\ x\in G^\vee,\ \phi\in\widehat{A(x)}\},
\end{equation}
such that
\begin{equation}
\widehat{G(K)}^\uni_\ds\longleftrightarrow G^\vee\backslash\{(x,\phi):\ x\in G^\vee \text{ elliptic},\ \phi\in\widehat{A(x)}\}.
\end{equation}
\end{theorem}
The parameterization of $\widehat{G(K)}^\uni_\ds$ was also obtained independently in the case when $G$ is split exceptional 
in \cite{Re3}, and when $G=SO(2n+1)$ in \cite{Wa}. It also follows from \cite[Theorem 3.4, Proposition 3.11]{O4}, in 
view of results of Slooten \cite{Sl}, \cite{OS2}, and \cite{CK} (see the discussion in \cite[paragraph 3.3.2]{O4}).

\smallskip

If $x\in G^\vee,$ write the Jordan decomposition $x=su$, where $s$ is semisimple and $u$ is unipotent. Notice that $x$ is elliptic if and only if the centralizer $Z_{G^\vee}(s)$ is semisimple and $u\in Z_{G^\vee}(s)$ is a distinguished unipotent element in the sense of Bala-Carter \cite{Ca}. Let $\Phi_u: SL(2,\bC)\to Z_{G^\vee}(s)$ be a Lie homomorphism, mapping $\left(\begin{matrix} 1&1\\0&1\end{matrix}\right)\mapsto u$, and set
\begin{equation}\label{e:modified-s}
s'=s\Phi_u(\left(\begin{matrix} q^{1/2}&0\\0&q^{-1/2}\end{matrix}\right)).
\end{equation}
Then $\operatorname{Ad}(s')u=u^{q}.$ Fix a maximal torus
$T^\vee\subset G^\vee$. Without loss of generality, we may arrange
that $s,s'\in T^\vee.$ For every root $\al$ of $(G^\vee,T^\vee)$, let
$e_\al$ denote the corresponding character of $T^\vee.$ 
We fix a Haar measure $\mu$ on $G(K)$ such that $\mu(I)=1$ for an Iwahori subgroup $I$ of $G(K)$.
The formal 
degree of $\pi\in \widehat{G(K)}^\uni_\ds$ is known to be given as in
the following theorem. This result is a special case of 
the conjecture formulated in \cite{HII}. The $q$-part in the formula was conjectured in
\cite{HO1}, and proved in \cite{HO2} in the case of unipotent discrete series representations
with Whittaker vectors. The explicit form of the
constants that multiply the $q$-parts was conjectured in \cite{Re3},
where the formula was also verified for exceptional split
groups and all unipotent discrete series. Recently, this expression
of formal degrees for all unipotent discrete series of unramified simple 
$p$-adic groups was verified in \cite{O3,O4}; the method also relies on
corresponding results for discrete series of affine Hecke algebras
from \cite{OS2}, \cite{CKK}, and \cite{CO}.

\begin{theorem}\label{t:formal-degree}
Suppose $\pi_{x,\phi}\in \widehat{G(K)}^\uni_\ds$ is parameterized by the pair $(x,\phi)$, for an elliptic element $x=su\in G^\vee$. Let $s'$ be as in (\ref{e:modified-s}). The formal degree of $\pi$ equals
\begin{equation}
\hat\mu(\pi_{x,\phi})=\frac{\phi(1)}{|A(x)||Z(G^\vee)|} m_x(q),\quad \text{where } m_x(q)=q^\nu\displaystyle{\frac{\prod'_{\al}(e_\al(s')-1)}{\prod'_\al (q e_\al(s')-1)}}, 
\end{equation}
where $\prod'$ means that the zero factors are ignored,  the products
vary over all roots $\al$ of $(G^\vee,T^\vee)$, and $\nu$ is the
number of positive roots.
\end{theorem}

\subsection{}\label{sec:4.5} To state our conjecture, fix $u$ a
representative of a unipotent conjugacy class in $G^\vee$ such that $u$ is the unipotent part of an elliptic element in $G^\vee.$ 
Set
\begin{equation}
\Gamma_u=Z_{G^\vee}(u)/Z_{G^\vee}(u)^0Z(G),
\end{equation}
a finite group. Let $M(\Gamma_u)$, $M(\Gamma_u)'$, and $\{~.~\}$ be as in section \ref{sec:fourier}. Recall that the elements of $M(\Gamma_u)$ are $\Gamma_u$-orbits of pairs $(y,\rho)$, where $y$ is an element of $\Gamma_u$ and $\rho$ is an irreducible representation of $C_{\Gamma_u}(y).$

Let $\Sigma_u$ be the set of $Z_{G^\vee}(u)$-orbits on $$\{(s,\phi):\ s\in
Z_{G^\vee}(u)\text{ semisimple},\ \phi\in\widehat{A(su)}\}.$$
If $s$ is such that $x=su$ is elliptic, then according to Theorem \ref{t:DLL}, the pair $(s,\phi)$ parameterizes an element of $\widehat{G(K)}^\uni_\ds.$

Suppose $(s,\phi)\in\Sigma_u$ is given. Conjugating by $Z_{G^\vee}(u)^0$ if necessary, we may assume that $s\in T^\vee$. The natural inclusion $Z_{G^\vee}(su)\to Z_{G^\vee}(u)$ induces a map
$A(su)\to A(u)$. This map is well defined because we have assumed $G^\vee$ to be simply-connected, thus $Z_{G^\vee}(s)$ is connected, and therefore $Z_{G^\vee}(su)^0=Z_{G^\vee}(s)\cap Z_{G^\vee}(u)^0$. If we denote by $\bar s$ the image of $s$ in $A(u)$, it is clear that the image of the map lands in $C_{A(u)}(\bar s).$ Therefore we have a well-defined map
\begin{equation}\label{e:components}
A(su)\longrightarrow C_{\Gamma_u}(\bar s).
\end{equation}

\begin{remark}
When $u$ is distinguished, the map (\ref{e:components}) is an isomorphism, but not in general. For example, suppose $s$ is a semisimple element in $E_7$ whose centralizer $Z_{G^\vee}(s)$ is of type $A_3+A_3+A_1$, and $u$ is a regular unipotent in $Z_{G^\vee}(s)$. The conjugacy class of $u$ in $E_7$ is labelled $A_4+A_1$ in Bala-Carter notation. Then $A(u)=\bZ/2\bZ$ and $A(su)=\bZ/4\bZ$, see \cite[page 71]{Re3}.
\end{remark}

Suppose from now on that $u$ is distinguished. Then $\Sigma_u$ can be identified with $M(\Gamma_u)$ via the map $(s,\phi)\to (y,\rho),$ where $y=\bar s$ is the coset of $s$ in $\Gamma_u$, and $\rho$ is the irreducible representation of $C_{\Gamma_u}(y)$ corresponding to $\phi$ under the isomorphism (\ref{e:components}).
Therefore, the unique irreducible discrete series $\pi_{s,u,\phi}$ parameterized by $(s,u,\phi)\in \Sigma_u$ can be redenoted as $$\pi_{u,y,\rho},\text{ where }(y,\rho)\in M(\Gamma_u).$$
Now suppose the discrete series representation $(\pi_{x,\phi},V_{x,\phi})$ has nonzero vectors under the action of the Iwahori subgroup $I$. The space of Iwahori fixed vectors $V_{x,\phi}^I$ is a discrete series module for the Iwahori-Hecke algebra $\C H(G(K),1_{\overline I})$. This algebra is obtained from the generic algebra $\C H(\C R,m)$ in Definition \ref{d:Hecke-generic} with $m(s)\equiv 1$, by specializing $\mathbf q$ to $q$, the cardinality of the residue field of $K$. The module $V_{s,\phi}^I$ itself is the specialization $\mathbf q=q$ of an $\C H(\C R,m)$-module  $Y_{x,\phi}(\mathbf q)$:
\begin{equation}
V_{x,\phi}^I=Y_{x,\phi}(\mathbf q)|_{\mathbf q=q}.
\end{equation}
Specializing instead to $\mathbf q=1$, one gets
\begin{equation}
\lim V_{x,\phi}^I:= Y_{x,\phi}(\mathbf q)|_{\mathbf q=1},\text{ a module for } W^e.
\end{equation}
Geometrically, this $W^e$-module is realized in the cohomology of the Springer fiber $\C B_x=\{B^\vee\text{ Borel subgroup of }G^\vee: x\in B^\vee\}$, i.e.,
\begin{equation}
\lim V_{x,\phi}^I=H^\bullet(\C B_x)^\phi\otimes\sgn, \text { as }W^e\text{-modules},
\end{equation}
see \cite[Corollary 8.1]{Re3} which relies on results of \cite{Ka} and \cite{L6}. 
When the discrete series representation is denoted by
$\pi_{u,y,\rho}$, let $\lim \pi_{u,y,\rho}^I$ denote the resulting
$W^e$-module. With this notation, we may now state our main
conjecture. It should be compared with Lusztig's formula (\ref{formal-finite}) for finite Lie groups.

\begin{conjecture}\label{conj-main} The formal degree of the unipotent discrete series $G(K)$-representation $\pi_{u,y,\rho}$, for a distinguished unipotent element $u$, is given by
\begin{equation}\label{e:conj-main}
\hat\mu(\pi_{u,y,\rho})=\frac{1}{|Z(G^\vee)|}\sum_{(y',\rho')\in M(\Gamma_u)'}\{(y,\rho),(y',\rho')\}~ F^e_{[\lim\pi_{u,y',\rho'}^I]},
\end{equation}
where $F^e_{[\lim\pi_{u,y',\rho'}^I]}$ is the elliptic fake degree from Definition \ref{d:ell-fake-affine}.
\end{conjecture}

\begin{remark}
To extend Conjecture \ref{conj-main} beyond the case when $u$ is distinguished, it appears that one needs to use for $\Gamma_u$ a larger group then $A(u)$. A likely candidate is given by the group (and the extension of the exotic Fourier transform) proposed recently by Lusztig in \cite{L5}. In section \ref{sec:sp4}, we present an example of non-distinguished $u$ where $A(u)$ is still sufficient for formal degrees.
\end{remark}

Conjecture \ref{conj-main} predicts another concrete interpretation of the $q$-part $m_x(q)$ of the formal degree from Theorem \ref{t:formal-degree}. Write the elliptic element $x=su$ with $s\in T^\vee_\iso.$
By \cite{Ka}, see \cite[Proposition 8.1]{Re3}, we know that
\begin{equation}\label{ind-coh}
H^\bullet(\C B_x)=\Ind_s(H^\bullet(\C B^{Z_{G^\vee}(s)}_u)),\text{ as }W^e\text{-modules},
\end{equation}
where $\C B^{Z_{G^\vee(s)}}_u$ is the Springer fiber of $u$ in the flag variety for $Z_{G^\vee}(s)$, and thus $H^\bullet(\C B^{Z_{G^\vee}(s)}_u)$ is a $W_s$-representation. By Definition \ref{d:ell-fake-affine}, this implies that
\begin{equation}\label{nonzero-fake}
 F^e_{[\lim\pi_{s,u,\phi}^I]},=\begin{cases}F^{W}_{[H^\bullet(\C
    B_u)^\phi\otimes\sgn]},&\text{if } s=1,~\phi\in\widehat{A(u)}_0,\\
0,&\text{otherwise},
\end{cases}
\end{equation}
where $F^W_{[~]}$ is the elliptic fake degree for the finite group $W$
as in Definition \ref{d:ell-fake-finite} and $\widehat{A(u)}_0$ is defined in (\ref{e:Springer-type}). 

This implies that the only entries of the exotic Fourier transform
matrix that contribute are of the form:
\begin{equation}
\{(y,\rho),(1,\rho')\}=\frac {\rho(1)}{|C_{\Gamma_u}(y)|}\rho'(y).
\end{equation}
Changing to the notation $(s,u,\phi)\in\Sigma_u$ for the unipotent
discrete series parameters, Conjecture \ref{conj-main} is equivalent with:
\begin{equation}\label{conj-equiv}
\begin{aligned}
\hat\mu(\pi_{s,u,\phi})&=\frac{\phi(1)}{|C_{\Gamma_u}(s)||Z(G^\vee)|}\sum_{\phi'\in\widehat{A(u)}_0}\phi'(s) F^{W}_{[H^\bullet(\C
    B_u)^{\phi'}\otimes\sgn]}\\
&=\frac{(1-q)^l\phi(1)}{|A(su)||Z(G^\vee)|}\left\langle\sum_{\phi'\in\widehat{A(u)}_0}\phi'(s)H^\bullet(\C
B_u)^{\phi'},\frac{1}{\det(1-q\cdot~)}\right\rangle_W^\el\\
&=\frac{(1-q)^l\phi(1)}{|A(su)||Z(G^\vee)|}\left\langle H^\bullet(\C
B_u)^s,\frac{1}{\det(1-q\cdot~)}\right\rangle_W^\el.
\end{aligned}
\end{equation}
Here, we think of $s$ as an element of $A(u)$. Comparing this formula against Theorem \ref{t:formal-degree}, we can make following remark.
\begin{remark}
Conjecture \ref{conj-main} predicts
that if $x=su$, the
$q$-part $m_x(q)$ of the formal degree for the unipotent discrete
series $\pi_{x,\phi}$ is
\begin{equation}\label{e:conj-second}
m_q(x)=(1-q)^l~\left\langle H^\bullet(\C B_u)^s,\frac
  {1}{\det(1-q\cdot~)}\right\rangle^\el_{W}.
\end{equation}
\end{remark}

\subsection{} The expectation in (\ref{e:conj-second}) has an
interesting implication. Suppose $G^\vee$ is a classical group and $x=su$ and $x'=s'u'$ are
representatives of distinct elliptic conjugacy classes in $G^\vee$
such that $u$ and $u'$ are distinguished. We
claim that (\ref{e:conj-second}) implies that
\begin{equation}\label{dist-packets}
m_x(q)=m_{x'}(q)\text{ if and only if } u=u'\text{ and }
\phi(s)=\phi(s')\text{ for all }\phi\in \widehat{A(u)}_0.
\end{equation}
Indeed, by Corollary \ref{c:ind-ell} and (\ref{e:conj-second}), 
\begin{equation}
m_x(q)=m_{x'}(q)\text{ if and only if } H^\bullet(\C
B_u)^s=H^\bullet(\C B_{u'})^{s'}\text{ in }\overline R(W).
\end{equation}
By \cite{R} or, equivalently, by the homological results from
\cite{OS1} recalled in the next section, the set $\{H^\bullet(\C B_u)^\phi\}$,
where $u$ varies over representatives of the the distinguished
unipotent classes and $\phi\in\widehat{A(u)}_0$, is orthonormal in
$\overline R(W).$ Then (\ref{dist-packets}) follows at once.

\begin{lemma}
Let $G^\vee$ be simple. Suppose the unipotent element $u$ is distinguished and $s,s'\in\Gamma_u$. If $\phi(s)=\phi(s')$ for all
$\phi\in\widehat{A(u)}_0$, then $s'$ is conjugate to $sz$, for some $z\in Z(G^\vee).$
\end{lemma}

\begin{proof}
Clearly the statement is equivalent with: if $s,s'\in A(u)$ such that
$\phi(s)=\phi(s')$ for all $\phi\in\widehat{A(u)}_0,$ then $s'$ is
conjugate to $s$ in $A(u).$ 

The nontrivial group $A(u)$ is one of $(\bZ/2\bZ)^k$, $S_3,$ $S_4,$
$S_5$. Suppose $A(u)=S_k$, $k=3,4,5$. The only $A(u)$-representation
which is not in $\widehat{A(u)}_0$ is $\sgn$. It is easy to check that
the other irreducible $A(u)$-representations separate the conjugacy
classes.

\smallskip

Now assume $A(u)=(\bZ/2\bZ)^k.$ If $G$ is exceptional, $k=0,1$ and
$\widehat{A(u)}=\widehat{A(u)}_0$, so the conclusion follows. 

\smallskip

Suppose $G$ is classical. In type $A$, we only have the regular
unipotent element $u$ for which $A(u)=\{1\}.$ 

\smallskip

If $G^\vee=Sp(2n)$, the
distinguished unipotent element $u$ is parameterized by partitions 
$(2a_1,2a_2,\dots,2a_m)$ of $2n$, with $a_1<a_2<\dots<a_m$,
$\Gamma_u=(\bZ/2\bZ)^m$ and 
$A(u)=\Gamma_u/(\bZ/2\bZ)_\triangle\cong(\bZ/2\bZ)^{m-1}.$ By allowing $a_1$ to be $0$, we may assume
that $m=2t+1$ is odd. The irreducible representations of $\Gamma_u$
(or equivalently $A(u)$) of
Springer type are in one-to-one correspondence with certain combinatorial objects, called symbols, introduced in \cite[\S11-13]{L7}, simplifying and generalizing earlier work of Shoji \cite{Sh}.
In our particular case, these symbols are of the form
\begin{equation}\label{S-symbol}
\left(\begin{matrix} b_1 & &b_2 && b_3 &&\dots &&b_{t+1}\\
&b_{t+2} &&b_{t+3}&&\dots &&b_{2t+1}\\
\end{matrix}\right),
\end{equation}
where $b_1<b_2<\dots<b_{t+1}$, $b_{t+2}<b_{t+3}<\dots<b_{2t+1}$, and 
$$\{b_1,b_2,\dots,b_{2t+1}\}=\{a_j+(j-1): j=1,\dots,2t\}.$$
The symbol
\begin{equation}\label{triv-A}
\left(\begin{matrix} a_1 & &a_3+2 && a_5+4 &&\dots &&a_{2t+1}+2t\\
&a_{2}+1 &&a_{4}+3&&\dots &&a_{2t}+2t-1\\
\end{matrix}\right)
\end{equation}
corresponds to the trivial $\Gamma_u$-representation. Now given a symbol (\ref{S-symbol}), the corresponding $\Gamma_u$-representation
has the trivial $\bZ/2\bZ$-representation in the $j$-th position if the element
$a_j+(j-1)$ occurs in the same row as in (\ref{triv-A}),
otherwise it has the $\sgn$ in the $j$-th position. Notice that in
particular all the $\Gamma_u$-representations of the form
\begin{equation}\label{A-separate}
\triv\boxtimes\dots\boxtimes\triv\boxtimes\sgn\boxtimes\sgn\boxtimes\triv\boxtimes\dots\boxtimes\triv
\end{equation}
appear, just by flipping two consecutive entries in the opposite rows
of (\ref{triv-A}). But these representations separate $A(u).$

\smallskip

When $G^\vee$ is an odd orthogonal group, the discussion is completely
analogous with the $Sp(2n)$ case, and we skip the details.

\smallskip

Suppose $G^\vee$ is an even orthogonal group. The distinguished unipotent
classes are parameterized by $(2a_1+1,2a_2+1,\dots,2a_{2t}+1)$, where
$0\le a_1<a_2<\dots<a_{2t}$, for which $A(u)=(\bZ/2\bZ)^{2t-2}.$

The S-symbol corresponding to the trivial $A(u)$-representation is
\begin{equation}\label{triv-A-typeD}
\left(\begin{matrix} a_1 & a_3+2 & a_5+4 &\dots &a_{2t-1}+2t-2\\
a_{2}+1 &a_{4}+3&a_6+5&\dots &a_{2t}+2t-1\\
\end{matrix}\right),
\end{equation}
and all other symbols are obtained by flipping entries between the
rows (as in type $C$). There is one difference: two symbols which
only differ by flipping the (full) rows are identified. To fix this, we
require that $a_1$ is always in the top row. Then the discussion is as
in type $C$, and the same representations (\ref{A-separate}) separate $A(u)$.
\end{proof}

In conclusion, we have the following remark.

\begin{remark}
Conjecture (\ref{e:conj-second}) implies that, when $G$ is a classical
group, the $q$-part of the formal
degree determines uniquely (up to $Z(G^\vee)$) the L-packet of
unipotent discrete series of $G(K)$, at least when the unipotent
element $u$ is distinguished. This statement is known to hold for
exceptional groups by \cite{Re3}.
\end{remark}

 \subsection{} Suppose $G$ is of type $G_2$ and $u$ is a
 representative of the subregular distinguished unipotent class,
 labelled $G_2(a_1)$. Then $\Gamma_u=A(u)=S_3.$ The sets $M(S_3)$ and
 $M(S_3)'$ have cardinalities $8$ and $4$, respectively. 

There are $3$ semisimple elements in $T^\vee_\iso$, which we denote by
$s_0=1$, $s_1$, and $s_2.$ The corresponding centralizers in
$G^\vee=G_2$ have types $G_2$, $A_1\times \wti A_1,$ and $A_2$,
respectively.

There are $8$ unipotent discrete series parameterized by $u$. They are
divided into three L-packets as in Table \ref{t:G2-packets}.

\begin{table}[h]
\caption{Unipotent parameters for $u=G_2(a_1)$, $\Gamma_u=S_3$, in $G_2$\label{t:G2-packets}}
\begin{tabular}{|c|c|c|c|c|}
\hline
$Z_{G^\vee}(s)$ &$A(s,u)$ &$ \phi\in\widehat{A(s,u)}$ &$(y,\rho)\in
M(\Gamma_u)$ &Notes \\
\hline
\hline
$G_2$ &$S_3$ &$1$ &$(1,1)$ &Iwahori, generic real c.c.\\
\hline
&&$\refl$ &$(1,r)$ &Iwahori, nongeneric real c.c.\\
\hline
&&$\sgn$ &$(1,\epsilon)$ &supercuspidal $G_2[1]$\\
\hline
\hline
$A_1+\wti A_1$ &$\bZ/2\bZ$ &$1$ &$(g_2,1)$ &Iwahori, endoscopic
$A_1\times\wti A_1$\\
\hline
&&$\sgn$ &$(g_2,\epsilon)$ &supercuspidal $G_2[-1]$\\
\hline 
\hline
$A_2$ &$\bZ/3\bZ$ &$1$ &$(g_3,1)$ &Iwahori, endoscopic $A_2$\\
\hline
&&$\zeta$ &$(g_3,\theta)$ &supercuspidal $G_2[\zeta]$\\
\hline
&&$\zeta^2$ &$(g_3,\theta^2)$ &supercuspidal $G_2[\zeta^2]$\\

\hline
\end{tabular}
\end{table}

Only the discrete series corresponding to the first two lines have
nonzero elliptic fake degrees and these are given by the entries in
Table \ref{t:G2}. The entries of the Fourier matrix $\{~,~\}$ are, for
example, in
\cite[page 457]{Ca}. Multiplying the relevant $8\times 2$ of the
matrix with the $2\times 1$ vector of nonzero elliptic fake degrees,
we obtain the vector of $8$ formal degrees from Table
\ref{t:G2-formal}, compare with \cite[section 7]{Re-Iw}.

\begin{table}[h]
\caption{Formal degrees, $u=G_2(a_1)$, $\Gamma_u=S_3$, in $G_2$\label{t:G2-formal}}
\begin{tabular}{|c|c|c|}
\hline
$Z_{G^\vee}(s)$ &$ \phi\in\widehat{A(s,u)}$ &Formal degree\\
\hline
\hline
$G_2$ &$1$ &$\frac 16 \frac{q(1-q)^2}{\Phi_2^2\Phi_3}$\\
\hline
&$\refl$ &$\frac 13 \frac{q(1-q)^2}{\Phi_2^2\Phi_3}$\\
\hline
&$\sgn$ &$\frac 16 \frac{q(1-q)^2}{\Phi_2^2\Phi_3}$\\
\hline
\hline
$A_1+\wti A_1$ &$1$ &$\frac 12 \frac{q(1-q)^2}{\Phi_2\Phi_6}$\\
\hline
&$\sgn$ &$\frac 12 \frac{q(1-q)^2}{\Phi_2\Phi_6}$\\
\hline 
\hline
$A_2$ &$1$ &$\frac 13 \frac{q(1-q)^2}{\Phi_3\Phi_6}$\\
\hline
&$\zeta$ &$\frac 13 \frac{q(1-q)^2}{\Phi_3\Phi_6}$\\
\hline
&$\zeta^2$ &$\frac 13 \frac{q(1-q)^2}{\Phi_3\Phi_6}$\\

\hline
\end{tabular}
\end{table}

\begin{remark}
In the same way, using the explicit tables of elliptic fake degrees
from section \ref{sec:ell}, we verified Conjecture \ref{conj-main} for
all exceptional simple $p$-adic groups, when the unipotent element $u$
is distinguished, via a 
comparison with the known expressions for formal degrees from
\cite{Re3}. 
\end{remark}

\subsection{}\label{sec:sp4} We offer another interesting example. Suppose
$G^\vee=Sp(4,\bC)$. Let $s\in G^\vee$ be a semisimple element such
that $Z_{G^\vee}(s)=SL(2,\bC)\times SL(2,\bC)$, and $u$ be a regular
unipotent element in $Z_{G^\vee}(s).$ Then $Z_{G^\vee}(u)\cong
O(2,\bC)$ and $\Gamma_u=\bZ/2\bZ.$ The sets $M(\bZ/2\bZ)$ and
$M(\bZ/2\bZ)'$ have cardinalities $4$ and $2$, respectively. If we
denote by $\tau$ the nontrivial element in $\Gamma_u$, the Fourier transform matrix
is:
\begin{equation}\label{FT:Z2}
\begin{tabular}{|c|c|c|c|c|}
\hline
$\bZ/2\bZ$&$(1,1)$ &$(1,\epsilon)$ &$(\tau,1)$ &$(\tau,\epsilon)$\\
\hline
$(1,1)$ &$1/2$ &$1/2$ &$1/2$ &$1/2$\\
\hline
$(1,\epsilon)$ &$1/2$ &$1/2$ &$-1/2$ &$-1/2$\\
\hline
$(\tau,1)$ &$1/2$ &$-1/2$ &$1/2$ &$-1/2$\\
\hline
$(\tau,\epsilon)$ &$1/2$ &$-1/2$ &$-1/2$ &$1/2$\\
\hline
\end{tabular}
\end{equation}
There are four elliptic unipotent representations attached to $u$, as
listed in Table \ref{t:C2-packets}. There are two discrete series
representations, one of which is supercuspidal, 
attached to $(\tau,1),(\tau,\epsilon)\in M(\bZ/2\bZ)$ in the table.
\begin{table}[h]
\caption{Unipotent elliptic parameters for $u=(2,2)$, $\Gamma_u=\bZ/2\bZ$, in $Sp(4,\bC)$\label{t:C2-packets}}
\begin{tabular}{|c|c|c|c|c|}
\hline
$Z_{G^\vee}(s)$ &$A(s,u)$ &$ \phi\in\widehat{A(s,u)}$ &$(y,\rho)\in
M(\Gamma_u)$ &Notes \\
\hline
\hline
$Sp(4,\bC)$ &$\bZ/2\bZ$ &$1$ &$(1,1)$ &Iwahori, generic real c.c.\\
\hline
&&$\sgn$ &$(1,\epsilon)$ &Iwahori, nongeneric real c.c.\\
\hline
\hline
$SL(2)\times SL(2)$ &$\bZ/2\bZ$ &$1$ &$(\tau,1)$ &Iwahori, endoscopic
$SL(2)\times SL(2)$\\
\hline
&&$\sgn$ &$(\tau,\epsilon)$ &supercuspidal\\
\hline
\end{tabular}
\end{table}
The two nonzero elliptic fake degrees are
\begin{equation}
F^e_{[\lim \pi^I_{u,1,1}]}=\frac {q(1-q)^2}{(1+q^2)(1+q)^2} \text{ and } F^e_{[\lim \pi^I_{u,1,\epsilon}]}=- \frac {q(1-q)^2}{(1+q^2)(1+q)^2}. 
\end{equation}
Then (\ref{e:conj-main}) gives 
\begin{equation}
\hat\mu(\pi_{u,1,1}) =\hat\mu(\pi_{u,1,\epsilon})=0 \text{ and }
\hat\mu(\pi_{u,\tau,1})= \hat\mu(\pi_{u,\tau,\epsilon})=\frac 12 \frac{q(1-q)^2}{(1+q^2)(1+q)^2}.
\end{equation}
This is consistent: $\mu(\pi_{u,1,1})$ and $\mu(\pi_{u,\tau,1})$ are not
discrete series representations, so the Plancherel measure should be zero, while the
results for
formal degrees for the two remaining discrete series can be checked
against \cite[Proposition 7.3]{Re-Iw}.

\section{Formal degrees in the Iwahori case}\label{sec:Iwahori}

In this section, we devise an approach towards verifying Conjecture
\ref{conj-main} directly, i.e., without refering to Theorem \ref{t:formal-degree}, at least for the discrete series with Iwahori fixed
vectors, since in that case there is no difference between the formal
degrees for the discrete series of the group and the ones for the
Iwahori-Hecke algebra, see Example \ref{ex:Iwahori}.

\subsection{} Recall the generic affine Hecke algebra from Definition
\ref{d:Hecke-generic}, and its specialization $\C H(\C R,m).$ Now specialize further $\mathbf q$ to $q>1.$ The algebra $\CH=\C H(\C R,m)$ has a structure of normalized
Hilbert algebra studied in \cite{O}. Let $*:\CH\to\CH$ be the conjugate
linear anti-involution defined on the basis by $$N_w^*=N_{w^{-1}},$$ and
let $\tau:\CH\to\bC$ be the trace
$$\tau(N_w)=\delta_{w,1}.$$
The pairing $(x,y)=\tau(x^*y)$ is a positive definite hermitian form
on $\CH$ such that the basis $\{N_w:w\in W^e\}$ is orthonormal with
respect to it. Let $C^*_r(\C H)$ be the reduced $C^*$-algebra
completion of $\C H.$ By \cite[Theorem 2.25]{O}, the abstract
Plancherel formula holds, i.e., there exists a unique positive Borel measure
$\hat\mu$ (the Plancherel measure of $\C H$) on $\widehat{C^*_r(\CH)}$ such that 
\begin{equation}\label{e:Plancherel}
\tau(x)=\int_{\widehat{C^*_r(\CH)}} \tr \pi(x) d\hat\mu(\pi),\quad x\in\CH.
\end{equation}

\begin{definition}
A simple $\C H$-module $\pi$ is called a tempered module if $\pi$ can be extended to a $C^*_r(\C H)$-module, i.e., if $\pi$ occurs in the support of the Plancherel measure $\hat\mu$ (\ref{e:Plancherel}). 

It is called a discrete series module
if $\hat\mu(\{\pi\})>0.$ The scalar $\hat\mu_\pi:=\hat\mu(\{\pi\})$ is
called the formal degree of the discrete series $\pi.$
\end{definition}

\subsection{} We present from \cite{OS1} the homological formula for
computing $\hat\mu_\pi$. Assume from now on that the root datum $\C R$
is semisimple, i.e., that $\Omega$ is finite, or else there are no
discrete series representations.

Let $E=X\otimes_\bZ \bR$, and let $A_\emptyset$ be the fundamental
alcove for the action of $W^e$ on $E$. The action of $\Omega$
preserves $A_\emptyset$. The facets of $A_\emptyset$ are in one-to-one
correspondence with subsets $J\subset F^a$:
$$A_J:=\{v\in E: \langle v, \mathbf a\rangle=0,\forall\mathbf a\in J,\
\langle v,\mathbf a'\rangle>0,\forall \mathbf a'\in F^a\setminus J\}.
$$ 
Let $\Omega_J$ denote the stabilizer in $\Omega$ of $A_J$ and $W_J$ be
the subgroup of $W^a$ generated by $\{s_{\mathbf a}:\mathbf a\in J\}.$
Define
\begin{equation}
\C H(\C R,J,m)=\C H(W_J,m)\rtimes\Omega_J.
\end{equation}
Here, $\C H(W_J,m)$ is the finite Hecke algebra generated by $W_J$ and
with parameters obtained from the restriction of $m$ to $J.$ By
\cite[Lemma 1.4]{OS1}, the algebra $\C H(\C R,J,m)$ is semisimple for
all $J$. For every irreducible $\C H(\C R,J,m)$-module $\sigma$, let
$e_\sigma\in \C H(\C R,J,m)$ denote the corresponding primitive
central idempotent, and let $\dim\sigma$ be the dimension of
$\sigma.$ Let $\overline S^a$ denote a set of representatives for the
orbits of $\Omega$ on $S^a.$ Finally, let $\ep_J$ denote the
orientation character of $\Omega_J$, i.e., the determinant of the
linear action of $\Omega_J$ on $E/\text{span}(A_J).$

\begin{definition}[{\cite[(3.19)]{OS1}}]
The Euler-Poincar\'e function $f_\EP^\pi$ associated to an irreducible
$\C H(\C R,m)$-module $(\pi,V)$ is 
\begin{equation}
f_\EP^\pi=\sum_{J\in \overline S^a}
(-1)^{|J|}\sum_{\sigma\in\Irr\CH(\CR,J,m)}\frac {[\pi\otimes\ep_J:\sigma]}{\dim\sigma} e_\sigma,
\end{equation}
where $[\pi\otimes\ep_J:\sigma]$ denotes the multiplicity of $\sigma$ in $\pi|_{\C H(\C R,J,m)}\otimes\ep_J.$
\end{definition}

The functions $f_\EP^\pi$ play an essential role in the elliptic representation theory of $\C H$ as we recall next.

\subsection{} Let $\CH$-mod denote the category of finite dimensional $\CH$-modules, $\widehat\CH$ the simple objects in $\CH$-mod, $\widehat\CH_\ds$ the set of simple discrete series modules. Let $\C R(\C H)$ denote the Grothendieck group of finite dimensional $\C H$-modules. The Euler-Poincar\' e pairing on $\C R(\C H)$ is
\begin{equation}\label{e:EP}
\langle~,~\rangle^\EP_\CH: \C R(\CH)\times \C R(\CH),\ \langle\pi,\pi'\rangle^\EP_\CH=\sum_{i\ge 0}(-1)^i\dim\Ext_\CH^i(\pi,\pi'),\ \pi,\pi'\in\C H\text{-mod}.
\end{equation}
This is well-defined, since $\C H$-mod has finite cohomological dimension \cite[Proposition 2.4]{OS1}.

\begin{theorem}[{\cite[Proposition 3.6, Theorem
    3.8]{OS1}}]\label{t:EP}\ 
\begin{enumerate}
\item If $\pi,\pi'\in\CH$-mod, then $\langle\pi,\pi'\rangle^\EP_\CH=\tr\pi(f_\EP^{\pi'}).$
\item Suppose $\pi'$ is a simple tempered $\CH$-module and $\pi\in\widehat\CH_\ds.$ Then
\begin{equation}
\Ext_\CH^i(\pi,\pi')=\begin{cases}\bC,&\text{if }\pi\cong\pi'\text{ and }i=0,\\ 0,&\text{otherwise.}\end{cases}
\end{equation}
\end{enumerate}
\end{theorem}
Using Theorem \ref{t:EP} and plugging in $f^\pi_\EP$ in (\ref{e:Plancherel}), one immediately obtains an explicit formula for the formal degree of $\pi\in\widehat\CH_\ds$, cf. \cite[Theorem 4.3]{OS2}. To state it, let $\hat\mu^f_\sigma$ be the formal degree of the simple module $\sigma$ for the finite (semisimple) algebra $ \CH(\CR,J,m)$. Then:
\begin{equation}\label{formal-homological}
\hat\mu_\pi=\tau(f^\pi_\EP)=\sum_{J\in \overline S^a}
(-1)^{|J|}\sum_{\sigma\in\Irr\CH(\CR,J,m)}{[\pi\otimes\ep_J:\sigma]}\frac{\hat\mu^f_\sigma}{P_J(q,m)},
\end{equation}
where $P_J(q,m)$ is the Poincar\'e polynomial of $\CH(\CR,J,m).$

\subsection{} From now on, assume we are in the setting of an
Iwahori-Hecke algebra of a split group. For simplicity, we also assume
that $W^e=W^a,$ i.e., the $p$-adic group is simply connected. In our notation, this is $\C
H=\C H(\C R,1).$ As it is well-known, the simple modules of the finite
Hecke algebras $\CH(\CR,J,1)$ are in one-to-one correspondence with
irreducible representations of the finite Weyl group $W_J.$ If
$\delta\in \widehat W_J$, write $d_\delta(q)$ for the generic degree
of $\delta$ (this corresponds to $\hat\mu^f_\sigma$ in
(\ref{formal-homological}). 

As in section \ref{sec:4.5}, write $\lim\pi\in R(W^a)$ for the ``limit''
$\mathbf q\to 1$ of the $\C H$-module $\pi.$

We may then rewrite
(\ref{formal-homological}) as:
\begin{equation}\label{formal-homological-2}
\hat\mu_\pi=\sum_{J\subset S^a}
(-1)^{|J|}\sum_{\delta\in\widehat W_J}{[\lim\pi:\delta]}\frac{d_\delta(q)}{P_J(q)}.
\end{equation}

We need a lemma first. 

\begin{lemma}\label{l:elliptic-affine}
Suppose $C$ is an elliptic conjugacy class in $W^a$. Then there exists one and only one maximal $J\subsetneq S^a$ such that $C\cap W_J\neq\emptyset,$ and in this case $C\cap W_J$ forms a single elliptic $W_J$-conjugacy class.
\end{lemma}

\begin{proof}
Let $t_xw\in W^a$ be an element of an elliptic class $C$. Suppose $e\in E$ is a fixed point for
$t_x w$, i.e., $(t_x w)(e)=e$, which is equivalent with
$x=(1-w)e$. Since $w\in W$ is necessarily elliptic, $1-w$ is
invertible, and therefore $e=(1-w)^{-1}x$ is a unique fixed
point. Conjugating $t_xw$ if necessary, we may assume that $e$ is in the closure of the
fundamental alcove, and therefore $t_wx\in Z_{W^a}(e)$, a parahoric
subgroup. Since $w\in W$ is elliptic, it is necessary that
$Z_{W^a}(e)=W_J$, for $J$ maximal, in other words, that $e$ is a
vertex of the fundamental alcove. The fact that $C\cap W_J$ is a
unique $W_J$-conjugacy class follows from the uniqueness of the fixed
point: if $(t_xw)e=e$ and $g(t_x w)g^{-1} e=e$ for some $g\in W^a$,
then $g^{-1}e$ is also a fixed point of $t_x w$, and therefore
$g^{-1}e=e,$ so $g\in W_J.$
\end{proof}

In light of Lemma \ref{l:elliptic-affine}, if $C_J$ is an elliptic
conjugacy class of $W_J$, where $J\subsetneq S^a$ is maximal, then we
may denote the unique elliptic conjugacy class in $W^a$ that meets
$C_J$ by $C_J^a.$

\begin{definition}
Define the class function $\nu: W^a\to
\bQ(q),$ by
\begin{equation}
\nu(C)=\begin{cases}0,&\text{if }C\text{ is not elliptic},\\
(-1)^l\displaystyle{\sum_{\delta\in\widehat W_J}\delta(C\cap
W_J)\frac{d_\delta(q)}{P_J(q)}},&\text{if }C=C_J^a \text{ is elliptic}.
\end{cases}
\end{equation}
\end{definition}

\begin{proposition}\label{p:new-formal}
For every discrete series $\pi\in\widehat{\C H}_\ds$, the formal
degree equals
\begin{equation}
\hat\mu_\pi=\langle\lim\pi,\nu\rangle^\el_{W^a},
\end{equation}
where $\langle~,~\rangle^\el_{W^a}$ is the affine elliptic pairing
from (\ref{affine-ell-pair}).
\end{proposition}

\begin{proof}
We have
\begin{equation*}
\begin{aligned}
\langle\lim\pi,\nu\rangle^\el_{W^a}&=\sum_{C\text{ elliptic in }W^a}
\lim\pi(C)\nu(C)\mu_\el(C)\ \text{(by the definition of the elliptic
  pairing)}\\
&=\sum_{\substack{J\subsetneq S^a\\ \text{ maximal}}}\sum_{\substack{C_J\subset
  W_J\\ \text{elliptic}}}\lim\pi(C_J^a)\nu(C_J^a)\frac{|C_J|}{|W_J|}\ \text{(by Lemma
  \ref{l:elliptic-affine} and definition of $\mu_\el$)}\\
&=(-1)^l\sum_{\substack{J\subsetneq S^a\\ \text{ maximal}}}\sum_{\substack{C_J\subset
  W_J\text{ elliptic}\\ \delta\in\widehat
  W_J}}\lim\pi(C_J)\delta(C_J)\frac{|C_J|}{|W_J|}\frac
{d_\delta(q)}{P_J(q)}\\
&=\sum_{\substack{J\subsetneq S^a}}(-1)^{|J|}\sum_{\substack{C_J\subset
  W_J\\ \delta\in\widehat
  W_J}}\lim\pi(C_J)\delta(C_J)\frac{|C_J|}{|W_J|}\frac
{d_\delta(q)}{P_J(q)}=\hat\mu_\pi \text{ (by
  (\ref{formal-homological-2})). }
\end{aligned}
\end{equation*}
The second to last step is justified as follows. Firstly, the set
$\{\lim\pi': \pi'\in R_{\ind}(\CH)\}$ separates all conjugacy classes
in all $W_J$'s with $J\subsetneq S^a$ not maximal. Secondly,
$\lim\pi'(C)=0$ for all $\pi'\in R_{\ind}(\CH)$ and all $C\subset W^a$
elliptic. Therefore, we can choose a
virtual properly induced module $\pi'$ such that 
\begin{align*}
\sum_{\substack{J\subsetneq S^a}}\sum_{\substack{C_J\subset
  W_J\\ \delta\in\widehat
  W_J}}\lim\pi'(C_J)\delta(C_J)\frac{|C_J|}{|W_J|}\frac
{d_\delta(q)}{P_J(q)}=&-\sum_{\substack{J\subsetneq S^a}}\sum_{\substack{C_J\subset
  W_J\\ \delta\in\widehat
  W_J}}\lim\pi(C_J)\delta(C_J)\frac{|C_J|}{|W_J|}\frac
{d_\delta(q)}{P_J(q)}\\
&+\sum_{C\text{ elliptic in }W^a}
\lim\pi(C)\nu(C)\mu_\el(C).
\end{align*}
Notice that the left hand side of the above equality is exactly
$\hat\mu_{\pi'}$, and this equals $0$, since $\pi'\in R_{\ind}(\CH).$
\end{proof}

\subsection{} We use the relation between formal degrees and fake degrees for
representations of finite Weyl groups from section
\ref{sec:fourier}. 

Let $\C W$ be a finite Weyl group. Using the exotic Fourier transform, we define a pairing
\begin{equation}\label{exotic-pair-Groth}
\{~,~\}: R(\C W)_\bC\times R(\C W)_\bC\to \bC,
\end{equation}
as follows.
 If $\delta,\delta'\in\widehat {\C W}$ are given, set
$\{\delta,\delta'\}=0$ if they are not in the same
family. Otherwise, let $\{\delta,\delta'\}$ be the entry in the
exotic Fourier transform $\{~,~\}$ for the pair of elements in
$M(\Gamma)'$ that parameterize $\delta$ and $\delta'$. Extend $\{~,~\}$ bilinearly to a pairing on $R(\C W)_\bC$. 

In particular, this definition applies to $\C W= W_J$, $J\subsetneq S^a$, in which case we denote the pairing by $\{~,~\}^J$ to emphasize the dependence on $J$.

 By
(\ref{formal-finite}), we have 
\begin{equation}\label{delta/P}
\begin{aligned}
\frac{d_\delta(q)}{P_J(q)}&=\sum_{\delta'\in\widehat
  W_J}\{\delta,\delta'\}^J
\frac{f_{\delta'}(q)}{P_J(q)}=(1-q)^l\sum_{\delta'\in\widehat W_J}
\{\delta,\delta'\}^J~\left\langle\delta',\frac
  1{\det(1-q\cdot~)}\right\rangle_{W_J}\\
&=(1-q)^l\sum_{C_J'\subset W_J}\sum_{\delta'\in\widehat W_J}
\{\delta,\delta'\}^J~ \delta'(C_J')~\frac
  1{\det(1-qC_J')}\frac{|C_J'|}{|W_J|};
\end{aligned}
\end{equation}
here $C_J'$ ranges over the conjugacy classes in $W_J.$

\subsection{} Now suppose again that $\C W$ is an arbitrary finite Weyl group. Define the linear
map 
\begin{equation}
\EF: R(\C W)_\bC\to R(\C W)_\bC,\quad \EF(\chi)=\sum_{\delta'\in \widehat{\C W}} \{\chi,\delta'\}\delta'.
\end{equation}

We expect that $\EF$ takes induced characters to induced characters, i.e.,
 \begin{equation}\label{ell-ind}
\EF(R_{\ind}(\C W))\subset R_{\ind}(\C W).
\end{equation}
 A stronger conjecture is
  that the Fourier transform commutes with parabolic induction
  $\ind_L$, where $\C W_L$ is a proper parabolic subgroup of $\C W$, i.e.,
\begin{equation}\label{Fourier-ind}
\ind_L(\EF^L(\delta))=\EF(\ind_L(\delta)),\ \delta\in R(\C W_L).
\end{equation}
For example, if $\delta$ is a left cell representation of $\C W_L$, then
by \cite[Theorem 12.2]{L1}, $\EF^L(\delta)=\delta$. Moreover, by
\cite[Proposition 3.15]{BV}, $\ind_L(\delta)$ is a sum of left cell
representations, and so (\ref{Fourier-ind}) holds in this case. In
particular, if $\C W_L$ is a product of symmetric groups (such as in
$G_2$), (\ref{Fourier-ind}) is true.

\smallskip

Equation (\ref{Fourier-ind}) can be reformulated in terms of unipotent
representations of finite groups of Lie type as follows. Retain the
notation from section \ref{sec:fourier}. In particular, the unipotent
almost-characters $R_{\C E}$ are defined for each $\C E\in \widehat {\C
  W}.$ One can extend the definition by linearity so that $R_{\C E}$
makes sense for every $\C E\in R(\C W).$ We need to distinguish
between almost-characters for the group $G^F$ and almost characters
for a Levi subgroup $L^F$ of a parabolic subgroup  $P^F$, so write $R^{G^F}_{\C E}$, $\C E\in R(\C W)$ for the former, and
$R^{L^F}_{\phi}$, $\phi\in R(\C W_L)$ for the latter. Then
(\ref{Fourier-ind}) is equivalent with
\begin{equation}
R^{G^F}_{\Ind_{\C W_L}^{\C W}(\phi)}=\Ind_{P^F}^{G^F}(R^{L^F}_\phi),\quad
\phi\in\widehat {\C W}_L.
\end{equation}

\subsection{}Let $\EF^J: R(W_J)_\bC\to R(W_J)_\bC$ be the linear map
defined in the previous subsection, specialized to the case $\C
W=W_J$, $J\subsetneq S^a.$
Assuming that (\ref{ell-ind}) is true, 
\begin{equation}
\EF^J_\el=\EF^J|_{\Ell(W_J)}: \Ell(W_J)\to \Ell(W_J).
\end{equation}

We rewrite the formula in Proposition
\ref{p:new-formal} along the lines of (\ref{delta/P}). For simplicity
of notation, suppose $v\in \overline R(W^a)$ is given; this vector
will be later specialized to $\lim\pi.$ Denote the delta function at a
conjugacy class $C$ by $\one_{C}.$ Define
\begin{equation}
F^a_\el=\sum_{J \text{ max}} \sum_{C_J\subset W_J\text{ ell}}\frac
1{\det(1-qC_J)} \one_{C_J}.
\end{equation}
We have:
\begin{equation}
\begin{aligned}
\hat\mu_v&=\langle v,\nu\rangle^\el_{W^a}\\
&=(q-1)^l\sum_{\substack{J\subsetneq S^a\\ \text{ max}}}\sum_{\substack{C_J\subset
  W_J\\ \text{ell}}} v(C_J)\frac{|C_J|}{|W_J|}\sum_{\delta\in\widehat
W_J}\delta(C_J)\sum_{\substack{C_J'\subset W_J\\ \delta'\in\widehat
W_J}}[\delta,\delta']^J \delta'(C_J')\frac
1{\det(1-qC_J')}\frac{|C_J'|}{|W_J|}\\
&=(q-1)^l \sum_{\substack{C_J\subset
  W_J\\ \text{ell}}}\sum_{C_J'\subset W_J} \langle
v,\one_{C_J}\rangle^\el_{W^a} \left( \sum_{\delta,\delta'\in\widehat
    W_J} \delta(C_J)[\delta,\delta']^J\delta'(C_J')\right) \langle
\one_{C_J'}, F^a_\el\rangle^\el_{W^a}\\
&= (q-1)^l \sum_{\substack{C_J,C_J'\subset
  W_J\\ \text{ell}}}\langle
v,\one_{C_J}\rangle^\el_{W^a} ~\langle\one_{C_J}|\EF^J_\el|\one_{C_J'}\rangle~\langle
\one_{C_J'}, F^a_\el\rangle^\el_{W^a},\\
\end{aligned}
\end{equation}
where $\langle\one_{C_J}|\EF^J_\el|\one_{C_J'}\rangle$ is the
appropriate entry in the matrix of $\EF^J_\ell$ in the basis given by
the $\one_C$, where $C$ ranges over the elliptic classes in $W_J.$

Motivated by this formula, it makes sense to define the linear
isomorphism
\begin{equation}
\EF^a_\el=\bigoplus_{J\subset S^a\text{ max}}\  \EF^J_\el:
\Ell(W^a)\to \Ell(W^a).
\end{equation}
If we fix an orthogonal basis $\{v'\}$ of $\Ell(W^a)$ 
we may write 
\begin{equation}
\langle
\one_{C_J'}, F^a_\el\rangle^\el_{W^a}=\sum_{v'}\langle
\one_{C_J'},v'\rangle^\el_{W^a} F^a_{v',\el}, \text{ where }
F^a_{v',\el}=\frac 1{\langle v',v'\rangle^\el_{W^a}}\langle v', F^a_\el\rangle^\el_{W^a}.
\end{equation}

Then our calculation can be expressed as follows.

\begin{corollary} Assume that (\ref{Fourier-ind}) holds for every
  $J\subsetneq S^a$. The formal degree of $v=\lim\pi$ is given by
$$\hat\mu_v=\sum_{v'}\langle v|\EF^a_\el|v'\rangle F^a_{v',\el},$$
where $v'$ ranges over an orthogonal basis 
$B_\el(W^a)$ of $\overline R(W^a)$ (such bases exists, see \cite[Section 5]{COT}).
\end{corollary}

When the root datum is both simply-connected and adjoint,
we expect the following

\begin{conjecture}\label{conj-ell-a}
The matrix of $\EF^a_\el$ in the basis $B_\el(W^a)$ equals the
submatrix of the exotic Fourier transform with rows/columns
corresponding to the elements of $B_\el(W^a).$
\end{conjecture}

\subsection{} We illustrate the previous calculations in the case of
$G_2.$ Let $\al$ be the short simple root and $\beta$ the long simple
root. Let $\theta$ be the short highest root. The affine Dynkin diagram is
\begin{equation}
G_2: \quad a_0-a_1\equiv a_2,
\end{equation}
where $a_0=1-\theta^\vee$, $a_1=\al^\vee$, and $a_2=\beta^\vee.$ The
elliptic conjugacy classes in $W^a$, i.e., the ``arithmetic'' side, are as follows:

\begin{enumerate}
\item $J_0=\{a_1,a_2\}$ gives $C_1^a=C(s_1s_2)$, $C_2^a=C((s_1s_2)^2)$,
  and $C_3^a=C((s_1s_2)^3)$;
\item $J_1=\{a_0,a_2\}$ gives $C_4^a=C(s_0s_2)$;
\item $J_2=\{a_0,a_1\}$ gives $C_5^a=C(s_0s_1)$.
\end{enumerate}

On the dual, ``spectral'' side, there are three isolated points $t_i$ in
$T^\vee=\Hom(X,C^\times)$ with endoscopic groups and representatives
of elliptic tempered representations (in fact discrete series
representations) as follows: 
\begin{enumerate}
\item $t_0=1$, $W_{t_0}=W(G_2)$: $v_1=V_{\mathsf{St}}^{t_0}$, with
  $v_1|_W=\phi_{(1,6)} $, $v_2=V_{\mathsf{refl}}^{t_0}$, with
  $v_2|_W=\phi_{(1,6)}+\phi_{(2,1)}$, and 
$v_3=V_{{(1,3)''}}^{t_0}$, with $v_3|_W=\phi_{(1,3)}''$; 
\item $t_1=\exp(\frac 13\theta^\vee)$, $W_{t_1}=W(A_2)$, with
  $v_4=V_{A_2}^{t_1}=\Ind_{W_{t_1}\ltimes X}^{W^a}(\mathsf{St}\otimes t_1)$;
\item $t_2=\exp(\frac 12\psi^\vee)$, $W_{t_2}=W(A_1\times A_1)$, with
  $v_5=V_{A_1\times A_1}^{t_2}=\Ind_{W_{t_2}\ltimes
    X}^{W^a}(\mathsf{St}\otimes t_2).$
\end{enumerate}
Here $\mathsf{St}$ means the Steinberg module, and the notation for
$W(G_2)$-representations is as in \cite{Ca}. The basis
$\{v_1,\dots,v_5\}$ in $\overline R(W^a)$ is orthonormal (as it is
formed of characters of  discrete series modules). We record in Table
\ref{t:G2-char-ell} the characters of $v_i$ on the conjugacy classes
$C_j^a$.

\begin{table}[h]
\caption{Elliptic character table of affine Weyl group of type $G_2$\label{t:G2-char-ell}}
\begin{tabular}{|c|c|c|c|c|c|}
\hline
&$C_1^a$ &$C_2^a$ &$C_3^a$ &$C_4^a$ &$C_5^a$\\
\hline
$\mu_\el$ &$1/6$ &$1/6$ &$1/12$ &$1/4$ &$1/3$\\
\hline
$v_1$ &$1$ &$1$ &$1$ &$1$ &$1$ \\
\hline
 $v_2$ &$2$ &$0$ &$-1$ &$-1$ &$0$ \\
\hline
 $v_3$ &$-1$ &$1$ &$-1$ &$-1$ &$1$ \\
\hline
 $v_4$ &$0$ &$2$ &$0$ &$0$ &$-1$ \\
\hline
 $v_5$ &$0$ &$0$ &$3$ &$-1$ &$0$ \\
\hline
 \end{tabular}
\end{table}

A direct calculation gives 
\begin{equation}
\EF^{J_0}_\el=\left(\begin{matrix}1/6 &1/2 &1/3\\
    1/2 &1/2 &0\\ 2/3 &0 &1/3\end{matrix}\right),
\end{equation}
in the basis $\{\one_{C_1^a},\one_{C_2^a},\one_{C_3^a}\}.$ Moreover,
$\EF^{J_1}_\el=1$ and $\EF^{J_2}_\el=1.$ Conjugating $\EF^a_\el$ in
the basis of $\one_{C^a_j}$'s by
the character table \ref{t:G2-char-ell}, we find
\begin{equation}
\EF^a_\el=\left(\begin{matrix}1 &&&&\\ &1/6 &1/3 &1/3 &1/2\\ &1/3
    &2/3 &1/3 &0\\ &1/3 &-1/3 &2/3 &0\\ &1/2 &0 &0 &1/2\end{matrix}\right),
\end{equation}
in the basis $\{v_1,v_2,v_3,v_4,v_5\}.$ This confirms Conjecture
\ref{conj-ell-a} in the case of $G_2$.

\appendix

\section{Elliptic fake degrees for exceptional Weyl groups} \label{sec:exc}

For a finite Weyl group $W$ of exceptional type, we give explicit formulas
for elliptic fake degrees in
the following tables. For simplicity, denote by $X(u,\phi)=H^\bullet(\C B_u)^\phi$ the Springer representations of $W$ considered before. By \cite{COT}, an orthogonal basis of
$\overline R(W)$ in all of these cases is
given by the set of  $X(u,\phi)$ for: 
\begin{enumerate}
\item[(i)] representatives $u$ of the distinguished unipotent classes and $\phi\in\widehat{A(u)}_0$;
 \item[(ii)] representatives $u$ of the quasi-distinguished non-distinguished unipotent classes  and $\phi=1$.
 \end{enumerate}
All the vectors in these bases have elliptic norm $1$, with one
exception in $E_7$ and the quasi-distinguished nilpotent orbit
$A_4+A_1$ when the elliptic norm is $\sqrt 2.$

To simplify the entries in the table, we write:
\begin{equation}
F_{[\pi]}=(q-1)^l\frac {N_{[\pi]}}{\cyc(W)},
\end{equation}
where $\cyc(W)$ is the product of cyclotomic polynomials which appears
as the denominator of the simplified form of $F_{[\sgn]}$. 
The explicit list is in Table \ref{t:cyc}.

\begin{table}[h]
\caption{$\cyc(W)$\label{t:cyc}}
\begin{tabular}{|c|c|}
\hline
$W$ &$\cyc(W)$\\
\hline
$G_2$ & $\Phi_2^2\Phi_3\Phi_6$\\
$F_4$ & $\Phi_2^4\Phi_3^2\Phi_4^2\Phi_6^2\Phi_8\Phi_{12}$\\
$E_6$ & $\Phi_2^2\Phi_3^3\Phi_6^2\Phi_9\Phi_{12}$\\
$E_7$ & $\Phi_2^7\Phi_3^2\Phi_4^2\Phi_6^3\Phi_8\Phi_{10}\Phi_{12}\Phi_{14}\Phi_{18}$\\
$E_8$ &$\Phi_2^8\Phi_3^4\Phi_4^4\Phi_5^2\Phi_6^4\Phi_8^2\Phi_9\Phi_{10}^2\Phi_{12}^2\Phi_{14}\Phi_{15}\Phi_{18}\Phi_{20}\Phi_{24}\Phi_{30}$\\
\hline
\end{tabular}
\end{table}

\begin{table}[h]
\caption{$G_2$\label{t:G2}}
\begin{tabular}{|c|c|c|c|c|}
\hline
$e\in \C N$ &$A(e)$ &$ \widehat{A(e)}_0$
&$\wti\sigma(e,\psi)\in \widehat{\wti W}_\gen$
&$N_{[X(u,\phi)]}$\\
\hline
$G_2$ &$1$ &$1$ &$2_s$ &${\Phi_5}$\\
\hline
$G_2(a_1)$ &$S_3$ &$(3)$ &$2_{sss}$ &${q\Phi_3}$\\
          &&$(21)$ &$2_{ss}$ & ${-q^2}$\\
\hline
\end{tabular}
\end{table}

\begin{table}[h]
\caption{$F_4$\label{t:F4}}
\begin{tabular}{|c|c|c|c|c|c|}
\hline

$e\in \C N$ &$A(e)$ &$\widehat{A(e)}_0$
&$\widehat{\wti W}_\gen$
&$N_{[X(u,\phi)]}$\\

\hline
$F_4$ &$1$ &$1$ &$4_s$ &${\Phi_5\Phi_7\Phi_{11}}$\\
\hline
$F_4(a_1)$ &$S_2$ &$(2)$ &$12_s$ &${q\Phi_5\Phi_7\cdot p(F_4(a_1),(2))
  } $\\ 
         & &$(11)$  &$8_{sss}$ &${-q^3\Phi_5\Phi_7\Phi_8} $\\  
\hline
$F_4(a_2)$ &$S_2$ &$(2)$ &$24_s$ &${q^2\Phi_5\Phi_8\cdot p(F_4(a_2),(2))
  }$\\ 
          &&$(11)$  &$8_{ssss}$ &${-q^3\Phi_5\Phi_7\Phi_8}$\\
\hline
$F_4(a_3)$ &$S_4$ &$(4)$ &$8_{ss}$
&${q^4\Phi_3^2\Phi_5\Phi_8}$\\ 
        &  &$(31)$ &$12_{ss}$ &$ {-q^5\cdot p(F_4(a_3),(31))}$\\ 
         & &$(22)$ &$8_{s}$
         &${-q^5\Phi_8\cdot p(F_4(a_3),(22))}$\\ 
         & &$(211)$ &$4_{ss}$ &${q^6\Phi_5^2}$\\
\hline
\end{tabular}
\end{table}

$\mathbf F_4$:

\smallskip

$p(F_4(a_1),(2))=1+q+q^2+q^4+q^6+q^7+q^8$

\smallskip

$p(F_4(a_2),(2))=1+q+2q^2+q^3+q^4+q^5+2q^6+q^7+q^8$

\smallskip

$p(F_4(a_3),(31)=1+q+2q^2+3q^3+5q^4+5q^5+5q^6+3q^7+2q^8+q^9+q^{10}$

$p(F_4(a_3),(22))=1+2q+4q^2+5q^3+4q^4+2q^5+q^6$

\begin{table}[h]
\caption{$E_6$\label{t:E6}}
\begin{tabular}{|c|c|c|c|c|}
\hline
$e\in \C N$ &$A(e)$ &$\psi\in \widehat{A(e)}_0$
&$\wti\sigma(e,\psi)\in \widehat{\wti W}_\gen$
&$N_{[X(u,\phi)]}$\\

\hline
$E_6$ &$1$ &$1$ &$8_s$ &$\Phi_7\Phi_{11}$\\
\hline
$E_6(a_1)$ &$1$ &$1$ &$40_s$ &$q\Phi_5\Phi_7\Phi_{12}$\\
\hline
$E_6(a_3)$ &$S_2$ &$(2)$ &$120_s$ &$q^3\Phi_5\Phi_9$\\
           &&$(11)$ &$72_s$ &$-q^4\Phi_2^2\Phi_9$\\
\hline
$D_4(a_1)$ &$S_3$ &$(3)$ &$40_{ss}$ &$q^7\Phi_2^2$\\

   \hline
\end{tabular}
\end{table}

\begin{table}[h]
\caption{$E_7$\label{t:E7}}
\begin{tabular}{|c|c|c|c|c|}
\hline

$e\in \C N$ &$A(e)$ &$\widehat{A(e)}_0$
&$\widehat{\wti W}_\gen$
&$N_{[X(u,\phi)]}$\\

\hline
$E_7$ &$1$ &$1$ &$8_s*$ &$\Phi_{11}\Phi_{13}\Phi_{17}$\\
\hline
$E_7(a_1)$ &$1$ &$1$ &$48_s*$ &$q\Phi_5\Phi_8\Phi_{11}\Phi_{13}\Phi_{18}$\\
\hline
$E_7(a_2)$ &$1$ &$1$ &$168_s*$ &$q^2\Phi_7^2\Phi_{11}\Phi_{14}\Phi_{18}$\\
\hline
$E_7(a_3)$ &$S_2$ &$(2)$ &$280_s*$ &$q^3\Phi_5\Phi_7\Phi_{10}\Phi_{14}\cdot
p(E_7(a_3),(2))$\\
           &&$(11)$ &$112_s*$ &$-q^5\Phi_5\Phi_7\Phi_8\Phi_{14}\cdot p(E_7(a_3),(11))$\\
\hline
$E_7(a_4)$ &$S_2$ &$(2)$ &$720_s*$
&$q^5\Phi_5\Phi_8\Phi_{10}\Phi_{18}\cdot p(E_7(a_4),(2))$ \\
           &&$(11)$ &$120_s*$
           &$-q^6\Phi_5\Phi_{10}\Phi_{18}\cdot p(E_7(a_4),(11))$\\
\hline
$E_7(a_5)$ &$S_3$ &$(3)$ &$448_{s}*$ &$q^7\Phi_3^3\Phi_4^2\Phi_8\Phi_{12}\Phi_{14}$\\
         & &$(21)$ &$560_s*$ &$-q^8\Phi_5\Phi_8\Phi_{10}\Phi_{14}\cdot
         p(E_7(a_5),(21))$\\
         & &$(111)$ &$112_{ss}*$ &$q^9\Phi_3^2\Phi_7\Phi_8\Phi_{14}$\\
\hline
$E_6(a_1)$ &$S_2$ &$(2)$ &$512_s*$ &$q^4\Phi_3^2\Phi_4^3\Phi_8^2\Phi_{12}\Phi_{16}$\\
\hline
$A_4+A_1$ &$S_2$ &$(2)$ &$64_s*$ &$2q^{11}\Phi_3^2\Phi_4^2\Phi_8\Phi_{12}$\\
          
\hline
\end{tabular}
\end{table}

\

$\mathbf E_7$:

\smallskip

$p(E_7(a_3),(2))=1+2q+2q^2+q^3+q^4+q^5+q^6+q^7+q^8+q^9+2q^{10}+2q^{11}+q^{12}$

$p(E_7(a_3),(11))=1+q+q^2-q^4+q^6+q^7+q^8$

\smallskip

$p(E_7(a_4),(2))=1+q+2q^2+2q^3+2q^4+q^5+2q^6+2q^7+2q^8+q^9+q^{10}$

$p(E_7(a_4),(11))=1+2q+2q^2+2q^3+3q^4+3q^5+3q^6+3q^7+3q^8+2q^9+2q^{10}+2q^{11}+q^{12}$

\smallskip

$p(E_7(a_5),(21))=2+4q+5q^2+4q^3+2q^4$

\begin{table}[h]
\caption{$E_8$\label{t:E8}}
\begin{tabular}{|c|c|c|c|c|}
\hline
$e\in \C N$ &$A(e)$ &$\widehat{A(e)}_0$
&$\widehat{\wti W}_\gen$
&$N_{[X(u,\phi)]}$\\

\hline
$E_8$ &$1$ &$1$ &$16_s$ &$\Phi_{11}\Phi_{13}\Phi_{17}\Phi_{19}\Phi_{23}\Phi_{29}$\\
\hline

$E_8(a_1)$ &$1$ &$1$  &$112_s$ &$q\Phi_7\Phi_9\Phi_{11}\Phi_{13}\Phi_{14}\Phi_{17}\Phi_{19}\Phi_{23}\Phi_{30}$\\
\hline

$E_8(a_2)$ &$1$ &$1$ &$448_{ss}$ &$q^2\Phi_7\Phi_8^2\Phi_{11}^2\Phi_{13}\Phi_{14}\Phi_{17}\Phi_{19}\Phi_{24}\Phi_{30}$\\
\hline

$E_8(a_3)$&$S_2$  &$(2)$ &$1344_{ss}$
&$q^3\Phi_7\Phi_8^2\Phi_9\Phi_{11}\Phi_{13}\Phi_{14}\Phi_{17}\Phi_{18}\Phi_{24}\cdot
p(E_8(a_3),(2))$\\
&&$(11)$ &$320_s$ &$q^7
\Phi_8^2\Phi_9\Phi_{11}\Phi_{13}\Phi_{17}\Phi_{24}\cdot p(E_8(a_3),(11))$\\
\hline

$E_8(a_4)$ &$S_2$ &$(2)$ &$2016_s$ &$q^4\Phi_7\Phi_9^2\Phi_{11}\Phi_{13}\Phi_{14}\Phi_{18}\cdot p(E_8(a_4)(2))$\\
&&$(11)$ &$1680_s$
&$q^6\Phi_{7}\Phi_9^2\Phi_{11}\Phi_{13}\Phi_{14}\Phi_{18}\cdot p(E_8(a_4),(11))$\\
\hline

$E_8(a_5)$ &$S_2$ &$(2)$ &$5600_{sss}$
&$q^6\Phi_5^3\Phi_7\Phi_9\Phi_{10}^2\Phi_{11}\Phi_{14}\Phi_{15}\Phi_{20}\Phi_{30}\cdot
p(E8(a_5),(2))$\\
           &&$(11)$ &$2800_s$
           &$-q^7\Phi_5^3\Phi_7\Phi_9\Phi_{10}^2\Phi_{11}\Phi_{14}\Phi_{15}\Phi_{20}\Phi_{30}\cdot
           p(E_8(a_5),(11))$\\
\hline

$E_8(a_6)$ &$S_3$ &$(3)$ &$6480_s$ &$q^8\Phi_9^2\Phi_{18}\cdot p(E_8(a_6),(3))$\\
&&$(21)$ &$9072_s$
&$-q^9\Phi_7\Phi_9^2\Phi_{14}\Phi_{18}\Phi_{20}\cdot p(E_8(a_6),(21))$\\
&&$(111)$ &$2592_s$ &$q^{12}\Phi_9^2\Phi_{11}\Phi_{18}\cdot p(E_8(a_6),(111))$\\
\hline

$E_8(a_7)$ &$S_5$ &$(5)$
&$896_{s}$
&$q^{16}\Phi_3^4\Phi_5^2\Phi_7\Phi_8^2\Phi_9\Phi_{14}\Phi_{15}\Phi_{24}\cdot
p(E_8(a_7),(5))$  \\
&&$(41)$ &$2016_{sss}$ &$-q^{17}\Phi_9\Phi_{14}\Phi_{18}\cdot p(E_8(a_7),(41))$\\
&&$(32)$ &$2016_{ss}$ &$-q^{17}\Phi_7\Phi_9\Phi_{14}\Phi_{18}\cdot p(E_8(a_7),(32))$\\
&&$(311)$ &$1344_s$ &$q^{18}\Phi_7\Phi_8^2\Phi_9\Phi_{14}\Phi_{18}\Phi_{24}\cdot p(E_8(a_7),(311))$\\
&&$(221)$ &$1120_s$ &$q^{18}\Phi_9\Phi_{14}\cdot p(E_8(a_7),(221))$\\
&&$(2111)$ &$224_s$ &$-q^{19}\Phi_7^2\Phi_9\Phi_{14}\cdot p(E_8(a_7),(2111))$\\
\hline

$E_8(b_4)$ &$S_2$ &$(2)$ &$5600_{ss}$
&$q^5\Phi_5^3\Phi_7\Phi_{10}^2\Phi_{11}\Phi_{13}\Phi_{14}\Phi_{15}\Phi_{20}\Phi_{30}\cdot
p(E_8(b_4),(2))$\\
 &&$(11)$ &$800_s$ &
 $q^7\Phi_5^3\Phi_{10}^2\Phi_{11}\Phi_{13}\Phi_{15}\Phi_{20}\Phi_{30}\cdot
 p(E_8(b_4),(11))$\\
\hline

$E_8(b_5)$ &$S_3$ &$(3)$ &$6720_s$ &$q^7\Phi_7\Phi_8^2\Phi_9\Phi_{11}\Phi_{14}\Phi_{18}\Phi_{24}\cdot p(E_8(b_5),(3))$\\ 
&&$(21)$ &$7168_s$ &$-q^8\Phi_4^4\Phi_7\Phi_8^2\Phi_{11}\Phi_{12}^2\Phi_{14}\Phi_{20}\Phi_{24}\cdot p(E_8(b_5),(21))$\\
&&$(111)$ &$448_s$ &$q^{13}\Phi_7\Phi_8^2\Phi_9\Phi_{11}\Phi_{13}\Phi_{14}\Phi_{24}\cdot p(E_8(b_5),(111))$ \\
\hline

$E_8(b_6)$ &$S_3$ &$(3)$ &$8400_s$ &$q^{10}\Phi_5^3\Phi_7\Phi_9\Phi_{10}^2\Phi_{14}\Phi_{15}\Phi_{18}\Phi_{20}\Phi_{30}\cdot p(E_8(b_6),(3))$\\
&&$(21)$ &$2800_{ss}$ &$-q^{11}\Phi_5^3\Phi_7\Phi_{10}^2\Phi_{14}\Phi_{15}\Phi_{20}\Phi_{30}\cdot p(E_8(b_6),(21))$\\
&&$(111)$ &$5600_{s}$ &$-q^{12}\Phi_5^3\Phi_7\Phi_9\Phi_{10}^2\Phi_{14}\Phi_{15}\Phi_{20}\Phi_{30}\cdot p(E_8(b_6),(111))$\\
\hline

$D_5+A_2$ &$S_2$ &$(2)$ &$4800_s$ &$q^{13}\Phi_3^2\Phi_5^4\Phi_8^2\Phi_9\Phi_{10}^2\Phi_{15}\Phi_{18}\Phi_{20}\Phi_{24}\Phi_{30}$\\
\hline

$D_7(a_1)$ &$S_2$ &$(2)$ &$11200_s$ &$q^9\Phi_3^2\Phi_5^3\Phi_7^2\Phi_8^2\Phi_9\Phi_{10}^2\Phi_{14}\Phi_{15}\Phi_{20}\Phi_{24}\Phi_{30}$\\

\hline

$D_7(a_2)$ &$S_2$ &$(2)$ &$7168_{ss}$ &$q^{12}\Phi_3^4\Phi_4^4\Phi_5^2\Phi_8^2\Phi_9\Phi_{12}^2\Phi_{15}\Phi_{16}\Phi_{20}\Phi_{24}$\\

\hline

$E_6(a_1)+A_1$ &$S_2$ &$(2)$ &$8192_s$ &$q^{11}\Phi_3^4\Phi_4^5\Phi_5^2\Phi_8^3\Phi_9\Phi_{12}^2\Phi_{15}\Phi_{16}\Phi_{20}\Phi_{24}$\\

\hline
\end{tabular}
\end{table}

\

$\mathbf E_8$:

\smallskip

$p(E_8(a_3),(2))=1+2
q+3 q^2+4 q^3+4 q^4+3 q^5+2
q^6+q^7+q^8+q^9+q^{10}+q^{11}+q^{12}+q^{13}+q^{14}+q^{15}+2 q^{16}+3
q^{17}+4 q^{18}+4 q^{19}+3 q^{20}+2 q^{21}+q^{22}$

$p(E_8(a_3),(11))=1+2 q+3 q^2+3 q^3+3
q^4+2 q^5+2 q^6+2 q^7+3 q^8+4 q^9+6 q^{10}+7 q^{11}+8 q^{12}+7 q^{13}+5 q^{14}+2 q^{15}+q^{16}+2 q^{17}+5 q^{18}+7
q^{19}+8 q^{20}+7 q^{21}+6 q^{22}+4 q^{23}+3 q^{24}+2 q^{25}+2 q^{26}+2 q^{27}+3 q^{28}+3 q^{29}+3 q^{30}+2 q^{31}+q^{32}$

\smallskip

$p(E_8(a_4),(2))=1+4 q+9 q^2+14 q^3+18 q^4+20 q^5+21 q^6+20 q^7+18 q^8+16 q^9+17 q^{10}+20 q^{11}+25 q^{12}+31 q^{13}+38 q^{14}+43 q^{15}+46 q^{16}+46 q^{17}+45 q^{18}+43 q^{19}+42 q^{20}+41 q^{21}+42 q^{22}+42 q^{23}+42 q^{24}+41 q^{25}+42 q^{26}+43 q^{27}+45 q^{28}+46 q^{29}+46 q^{30}+43 q^{31}+38 q^{32}+31 q^{33}+25 q^{34}+20 q^{35}+17 q^{36}+16 q^{37}+18 q^{38}+20 q^{39}+21 q^{40}+20 q^{41}+18 q^{42}+14 q^{43}+9 q^{44}+4 q^{45}+q^{46}$

$p(E_8(a_4),(11))=1+3 q+6 q^2+9 q^3+12 q^4+14 q^5+16 q^6+17 q^7+17 q^8+15 q^9+13 q^{10}+12 q^{11}+14 q^{12}+18 q^{13}+23 q^{14}+27 q^{15}+30 q^{16}+30 q^{17}+29 q^{18}+27 q^{19}+26 q^{20}+25 q^{21}+26 q^{22}+27 q^{23}+29 q^{24}+30 q^{25}+30 q^{26}+27 q^{27}+23 q^{28}+18 q^{29}+14 q^{30}+12 q^{31}+13 q^{32}+15 q^{33}+17 q^{34}+17 q^{35}+16 q^{36}+14 q^{37}+12 q^{38}+9 q^{39}+6 q^{40}+3 q^{41}+q^{42}$

\smallskip

$p(E_8(a_5),(2))=1+3 q+6 q^2+8 q^3+8 q^4+6 q^5+5 q^6+5 q^7+6 q^8+7 q^9+8 q^{10}+8 q^{11}+8 q^{12}+7 q^{13}+6 q^{14}+5 q^{15}+5 q^{16}+6 q^{17}+8 q^{18}+8 q^{19}+6 q^{20}+3 q^{21}+q^{22}$

$p(E_8(a_5),(11))=1+3 q+5 q^2+6 q^3+7 q^4+8 q^5+9 q^6+8 q^7+5 q^8+2 q^9+q^{10}+2 q^{11}+5 q^{12}+8 q^{13}+9 q^{14}+8 q^{15}+7 q^{16}+6 q^{17}+5 q^{18}+3 q^{19}+q^{20}$

\smallskip

$p(E_8(a_6),(3))=1+6 q+21 q^2+53 q^3+111 q^4+203 q^5+339 q^6+524 q^7+765 q^8+1064 q^9+1431 q^{10}+1867 q^{11}+2379 q^{12}+2962 q^{13}+3624 q^{14}+4354 q^{15}+5152 q^{16}+5996 q^{17}+6888 q^{18}+7807 q^{19}+8760 q^{20}+9726 q^{21}+10720 q^{22}+11718 q^{23}+12727 q^{24}+13709 q^{25}+14673 q^{26}+15581 q^{27}+16444 q^{28}+17224 q^{29}+17943 q^{30}+18568 q^{31}+19121 q^{32}+19560 q^{33}+19903 q^{34}+20105 q^{35}+20185 q^{36}+20105 q^{37}+19903 q^{38}+19560 q^{39}+19121 q^{40}+18568 q^{41}+17943 q^{42}+17224 q^{43}+16444 q^{44}+15581 q^{45}+14673 q^{46}+13709 q^{47}+12727 q^{48}+11718 q^{49}+10720 q^{50}+9726 q^{51}+8760 q^{52}+7807 q^{53}+6888 q^{54}+5996 q^{55}+5152 q^{56}+4354 q^{57}+3624 q^{58}+2962 q^{59}+2379 q^{60}+1867 q^{61}+1431 q^{62}+1064 q^{63}+765 q^{64}+524 q^{65}+339 q^{66}+203 q^{67}+111 q^{68}+53 q^{69}+21 q^{70}+6 q^{71}+q^{72}$

$p(E_8(a_6),(21))=1+5 q+17 q^2+42 q^3+86 q^4+150 q^5+237 q^6+345 q^7+476 q^8+625 q^9+792 q^{10}+969 q^{11}+1154 q^{12}+1335 q^{13}+1508 q^{14}+1662 q^{15}+1797 q^{16}+1904 q^{17}+1987 q^{18}+2042 q^{19}+2080 q^{20}+2102 q^{21}+2121 q^{22}+2133 q^{23}+2145 q^{24}+2147 q^{25}+2145 q^{26}+2133 q^{27}+2121 q^{28}+2102 q^{29}+2080 q^{30}+2042 q^{31}+1987 q^{32}+1904 q^{33}+1797 q^{34}+1662 q^{35}+1508 q^{36}+1335 q^{37}+1154 q^{38}+969 q^{39}+792 q^{40}+625 q^{41}+476 q^{42}+345 q^{43}+237 q^{44}+150 q^{45}+86 q^{46}+42 q^{47}+17 q^{48}+5 q^{49}+q^{50}$

$p(E_8(a_6),(111))=1+3 q+8 q^2+17 q^3+34 q^4+57 q^5+87 q^6+118 q^7+151 q^8+180 q^9+211 q^{10}+242 q^{11}+283 q^{12}+328 q^{13}+383 q^{14}+438 q^{15}+498 q^{16}+547 q^{17}+589 q^{18}+612 q^{19}+631 q^{20}+642 q^{21}+665 q^{22}+690 q^{23}+727 q^{24}+755 q^{25}+780 q^{26}+784 q^{27}+780 q^{28}+755 q^{29}+727 q^{30}+690 q^{31}+665 q^{32}+642 q^{33}+631 q^{34}+612 q^{35}+589 q^{36}+547 q^{37}+498 q^{38}+438 q^{39}+383 q^{40}+328 q^{41}+283 q^{42}+242 q^{43}+211 q^{44}+180 q^{45}+151 q^{46}+118 q^{47}+87 q^{48}+57 q^{49}+34 q^{50}+17 q^{51}+8 q^{52}+3 q^{53}+q^{54}$

\smallskip

$p(E_8(a_7),(5))=1+q+q^4+q^6+q^7+q^9+q^{10}+q^{12}+q^{15}+q^{16}$

 $p(E_8(a_7),(41))=1+5 q+17 q^2+46 q^3+106 q^4+215 q^5+396 q^6+673 q^7+1068 q^8+1597 q^9+2273 q^{10}+3104 q^{11}+4092 q^{12}+5225 q^{13}+6479 q^{14}+7818 q^{15}+9200 q^{16}+10576 q^{17}+11900 q^{18}+13129 q^{19}+14230 q^{20}+15178 q^{21}+15963 q^{22}+16584 q^{23}+17051 q^{24}+17374 q^{25}+17564 q^{26}+17626 q^{27}+17564 q^{28}+17374 q^{29}+17051 q^{30}+16584 q^{31}+15963 q^{32}+15178 q^{33}+14230 q^{34}+13129 q^{35}+11900 q^{36}+10576 q^{37}+9200 q^{38}+7818 q^{39}+6479 q^{40}+5225 q^{41}+4092 q^{42}+3104 q^{43}+2273 q^{44}+1597 q^{45}+1068 q^{46}+673 q^{47}+396 q^{48}+215 q^{49}+106 q^{50}+46 q^{51}+17 q^{52}+5 q^{53}+q^{54}$

$p(E_8(a_7),(32))=1+4 q+12 q^2+28 q^3+53 q^4+84 q^5+118 q^6+150 q^7+175 q^8+189 q^9+192 q^{10}+182 q^{11}+158 q^{12}+118 q^{13}+65 q^{14}+3 q^{15}-62 q^{16}-128 q^{17}-190 q^{18}-246 q^{19}-292 q^{20}-326 q^{21}-345 q^{22}-354 q^{23}-356 q^{24}-354 q^{25}-345 q^{26}-326 q^{27}-292 q^{28}-246 q^{29}-190 q^{30}-128 q^{31}-62 q^{32}+3 q^{33}+65 q^{34}+118 q^{35}+158 q^{36}+182 q^{37}+192 q^{38}+189 q^{39}+175 q^{40}+150 q^{41}+118 q^{42}+84 q^{43}+53 q^{44}+28 q^{45}+12 q^{46}+4 q^{47}+q^{48}$

$p(E_8(a_7),(311))=2+9 q+25 q^{2}+54 q^{3}+97 q^{4}+150 q^{5}+208 q^{6}+267 q^{7}+326 q^{8}+386 q^{9}+449 q^{10}+514 q^{11}+577 q^{12}+631 q^{13}+668 q^{14}+681 q^{15}+668 q^{16}+631 q^{17}+577 q^{18}+514 q^{19}+449 q^{20}+386 q^{21}+326 q^{22}+267 q^{23}+208 q^{24}+150 q^{25}+97 q^{26}+54 q^{27}+25 q^{28}+9 q^{29}+2 q^{30}$

$p(E_8(a_7),(221))=1+4 q+10 q^{2}+21 q^{3}+42 q^{4}+76 q^{5}+121 q^{6}+169 q^{7}+213 q^{8}+248 q^{9}+271 q^{10}+275 q^{11}+255 q^{12}+205 q^{13}+118 q^{14}-17 q^{15}-201 q^{16}-426 q^{17}-674 q^{18}-933 q^{19}-1194 q^{20}-1453 q^{21}-1700 q^{22}-1931 q^{23}-2141 q^{24}-2330 q^{25}-2493 q^{26}-2626 q^{27}-2714 q^{28}-2746 q^{29}-2714 q^{30}-2626 q^{31}-2493 q^{32}-2330 q^{33}-2141 q^{34}-1931 q^{35}-1700 q^{36}-1453 q^{37}-1194 q^{38}-933 q^{39}-674 q^{40}-426 q^{41}-201 q^{42}-17 q^{43}+118 q^{44}+205 q^{45}+255 q^{46}+275 q^{47}+271 q^{48}+248 q^{49}+213 q^{50}+169 q^{51}+121 q^{52}+76 q^{53}+42 q^{54}+21 q^{55}+10 q^{56}+4 q^{57}+q^{58}$

$p(E_8(a_7),(2111))=1+3 q+6 q^{2}+10 q^{3}+16 q^{4}+23 q^{5}+30 q^{6}+37 q^{7}+45 q^{8}+54 q^{9}+63 q^{10}+71 q^{11}+80 q^{12}+89 q^{13}+96 q^{14}+99 q^{15}+100 q^{16}+100 q^{17}+100 q^{18}+100 q^{19}+100 q^{20}+99 q^{21}+96 q^{22}+89 q^{23}+80 q^{24}+71 q^{25}+63 q^{26}+54 q^{27}+45 q^{28}+37 q^{29}+30 q^{30}+23 q^{31}+16 q^{32}+10 q^{33}+6 q^{34}+3 q^{35}+q^{36}$

\smallskip

$p(E_8(b_4),(2))=1+2 q+3 q^{2}+3 q^{3}+3 q^{4}+2 q^{5}+q^{6}+q^{7}+2 q^{8}+2 q^{9}+2 q^{10}+q^{11}+q^{12}+2 q^{13}+3 q^{14}+3 q^{15}+3 q^{16}+2 q^{17}+q^{18}$

$p(E_8(b_4),(11))=1+2 q+3 q^{2}+4 q^{3}+5 q^{4}+6 q^{5}+7 q^{6}+8 q^{7}+9 q^{8}+9 q^{9}+9 q^{10}+8 q^{11}+8 q^{12}+8 q^{13}+8 q^{14}+8 q^{15}+9 q^{16}+9 q^{17}+9 q^{18}+8 q^{19}+7 q^{20}+6 q^{21}+5 q^{22}+4 q^{23}+3 q^{24}+2 q^{25}+q^{26}$

\smallskip

$p(E_8(b_5),(3))=1+4 q+10 q^{2}+19 q^{3}+31 q^{4}+44 q^{5}+57 q^{6}+69 q^{7}+81 q^{8}+93 q^{9}+107 q^{10}+122 q^{11}+138 q^{12}+152 q^{13}+163 q^{14}+170 q^{15}+175 q^{16}+177 q^{17}+178 q^{18}+177 q^{19}+176 q^{20}+175 q^{21}+176 q^{22}+177 q^{23}+178 q^{24}+177 q^{25}+175 q^{26}+170 q^{27}+163 q^{28}+152 q^{29}+138 q^{30}+122 q^{31}+107 q^{32}+93 q^{33}+81 q^{34}+69 q^{35}+57 q^{36}+44 q^{37}+31 q^{38}+19 q^{39}+10 q^{40}+4 q^{41}+q^{42}$

$p(E_8(b_5),(21))=1+3 q+6 q^{2}+10 q^{3}+14 q^{4}+17 q^{5}+20 q^{6}+23 q^{7}+26 q^{8}+28 q^{9}+28 q^{10}+26 q^{11}+24 q^{12}+22 q^{13}+21 q^{14}+22 q^{15}+24 q^{16}+26 q^{17}+28 q^{18}+28 q^{19}+26 q^{20}+23 q^{21}+20 q^{22}+17 q^{23}+14 q^{24}+10 q^{25}+6 q^{26}+3 q^{27}+q^{28}$

$p(E_8(b_5),(111))=1+3 q+5 q^{2}+6 q^{3}+6 q^{4}+4 q^{5}+q^{6}-2 q^{7}-3 q^{8}+7 q^{10}+14 q^{11}+17 q^{12}+14 q^{13}+7 q^{14}-3 q^{16}-2 q^{17}+q^{18}+4 q^{19}+6 q^{20}+6 q^{21}+5 q^{22}+3 q^{23}+q^{24}$

\smallskip

$p(E_8(b_6),(3))=1+4 q+10 q^{2}+19 q^{3}+31 q^{4}+43 q^{5}+54 q^{6}+63 q^{7}+69 q^{8}+71 q^{9}+69 q^{10}+63 q^{11}+54 q^{12}+43 q^{13}+31 q^{14}+19 q^{15}+10 q^{16}+4 q^{17}+q^{18}$

$p(E_8(b_6),(21))=1+4 q+9 q^{2}+15 q^{3}+22 q^{4}+30 q^{5}+39 q^{6}+49 q^{7}+60 q^{8}+70 q^{9}+78 q^{10}+83 q^{11}+86 q^{12}+88 q^{13}+89 q^{14}+88 q^{15}+86 q^{16}+83 q^{17}+78 q^{18}+70 q^{19}+60 q^{20}+49 q^{21}+39 q^{22}+30 q^{23}+22 q^{24}+15 q^{25}+9 q^{26}+4 q^{27}+q^{28}$

$p(E_8(b_6),(111))=1+2 q+2 q^{2}-q^{3}-7 q^{4}-16 q^{5}-24 q^{6}-28 q^{7}-26 q^{8}-21 q^{9}-18 q^{10}-21 q^{11}-26 q^{12}-28 q^{13}-24 q^{14}-16 q^{15}-7 q^{16}-q^{17}+2 q^{18}+2 q^{19}+q^{20}$

\ifx\undefined\bysame
\newcommand{\bysame}{\leavevmode\hbox to3em{\hrulefill}\,}
\fi

\end{document}